\documentclass[review,onefignum,onetabnum,final]{siamart190516}

\usepackage{graphicx,epstopdf}
\usepackage{subfig}
\usepackage{enumitem} 
\usepackage{tikz}
\usepackage{titlesec}

\usepackage{breqn} 
\usepackage{fancyhdr}
\usepackage{algorithm}
\usepackage{algorithmic}
\usepackage{lipsum}
\usepackage{amsfonts}

\newtheorem{example}{Example}[section]

\newtheorem{remark}{Remark}

\usepackage{amssymb}
\usepackage{bbm}
\usepackage[numbers,sort&compress]{natbib}
\usepackage{hyperref}
\def\equationautorefname~#1\null{ (#1)\null}
\hypersetup{
  colorlinks = true,
  citecolor = blue
  }

\newcommand{\td}{\tilde}
\newcommand{\mbb}{\mathbb}
\newcommand{\mcal}{\mathcal}
\newcommand{\parSp}{\mcal P}		
\newcommand{\opn}{\operatorname}
\newcommand{\hsp}{\hspace{0.1cm}}
\newcommand{\hspB}{\hspace{0.3cm}}
\newcommand{\pd}{\partial}

\newcommand{\samZ}[1]{z^{(#1)}}

\newcommand{\idFull}{\mcal E_{full}}

\newcommand{\zRef}[1]{z_{*}^{(#1)}}

\newcommand{\shiftOp}[1]{\mcal T\left[{#1}\right]}

\DeclareMathOperator*{\argmin}{arg\,min}

\newcommand{\resFV}{\opn{Res}}

\usepackage{lmodern}
\usepackage{mathtools}
\usetikzlibrary{decorations.pathreplacing}
\usetikzlibrary{decorations.shapes}
\usetikzlibrary{calc}

\tikzset{decorate sep/.style 2 args=
{decorate,decoration={shape backgrounds,shape=circle,shape size=#1,shape sep=#2}}}

\usetikzlibrary{shapes.misc}
\tikzset{cross/.style={cross out, draw=black, minimum size=2*(#1-\pgflinewidth), inner sep=0pt, outer sep=0pt},
cross/.default={1pt}}

\title{Hyper-reduction for parametrized transport dominated problems via online-adaptive reduced meshes\thanks{Submitted to the editors xxxx\funding{N.S and S.G are supported by the German Federal Ministry for Economic Affairs and Energy (BMWi) in the joint project "MathEnergy - Mathematical Key Technologies for Evolving Energy Grids", sub-project: Model Order Reduction (Grant number: 0324019B).}}}
\author{Neeraj Sarna\thanks{Corresponding author, Max Planck Institute for Dynamics of Complex Technical Systems, Sandtorstr 1, 39106, Magdeburg, Germany, \email{sarna@mpi-magdeburg.mpg.de}}\and Sara Grundel\thanks{Max Planck Institute for Dynamics of Complex Technical Systems, Sandtorstr 1, 39106, Magdeburg, Germany, \email{grundel@mpi-magdeburg.mpg.de}}
}

\headers{Hyper-reduction via Online-Adaptive Reduced Mesh}{N. Sarna and S. Grundel}

\begin{document}
\nolinenumbers
\maketitle

\begin{abstract}
We propose an efficient residual minimization technique for the nonlinear model-order reduction of parameterized hyperbolic partial differential equations. Our nonlinear approximation space is a span of snapshots evaluated on a shifted spatial domain, and we compute our reduced approximation via residual minimization. To speed-up the residual minimization, we compute and minimize the residual on a (preferably small) subset of the mesh, the so-called reduced mesh. 
Due to the nonlinearity of our approximation space we show that, similar to the solution, the residual also exhibits transport-type behaviour. To account for this behaviour, we introduce online-adaptivity in the reduced mesh by "moving" it along the spatial domain with parameter dependent shifts. We also present an extension of our method to spatial transforms different from shifting. Numerical experiments showcase the effectiveness of our method and the inaccuracies resulting from a non-adaptive reduced mesh. 
\end{abstract}

\section{Introduction}
For some solution $u(x,z)\in\mbb R$, consider a parametrized partial differential equation (PDE) given as
\begin{gather}
\pd_t u(x,z) + \mcal L(u(x,z),z) = 0,\hspB\forall (x,t,z)\in\Omega\times D\times \mcal Z. \label{evo gen}
\end{gather}
Here, $x\in\Omega\subset \mbb R^d$ is a point in the space domain $\Omega$, $\mcal Z\subset \mbb R^{p+1}$ is some parameter domain, $t\in D$ is a point in the time-domain and $\mcal L$ is some spatial differential operator. For some finite time $T>0$, we include the time-domain $D:=[0,T]$ in the parameter domain $\mcal Z$ and express $\mcal Z$ as $\mcal Z = D\times \mcal P$, where $\mcal P\subset\mbb R^p$ is some additional parameter domain. The solution's dependency on $\mu\in\mcal P$ can encode, for instance, the change in the material properties, variation in length scales, changes in the background temperature, etc. The later sections of our article further elaborate on the relevance of $\mcal P$, we refer to \cite{RBBook} for additional examples. 

An exact solution to the above equation is often unavailable and one seeks a numerical approximation 
$$
u(\cdot,z)\approx u_N(\cdot,z)\in\mcal X_N,
$$
with $\mcal X_N$ being a $N$-dimensional finite-volume/element/difference-type space. We refer to $u_N$ as the full-order model (FOM). In a multi-query setting where solutions at several different parameter instances are needed, due to the high-dimensionality of $\mcal X_N$, computing a FOM is unaffordable. This motivates one to consider a reduced-order model (ROM).

A ROM splits the solution algorithm into an offline and an online phase and performs most of the expensive computations offline, thus making the online phase efficient. For some finite number of training parameters $\{\samZ{i}\}_{i=1,\dots,m}\subset\mcal Z$, the offline phase computes the set of solution snapshots $\{u_N(\cdot,\samZ{i})\}_{i}$ that are used by the online phase to efficiently approximate the FOM. If the number of snapshots required to reasonably approximate the FOM are sufficiently small, then one can expect a ROM to be more efficient than the FOM.  We refer to the review article \cite{PeterReview} and the later sections of our article for further details of the offline and the online phase.

This work focuses on parametrized hyperbolic PDEs. For such equations, there is ample numerical and analytical evidence indicating that a ROM based on a linear approximation space is inefficient/inaccurate---a large number of solution snapshots are required to reasonably approximate the FOM. The inefficiency arises from the poor approximability of the solution set $\{u(\cdot,z)\hsp :\hsp z\in\mcal Z\}$ (or the so-called solution manifold) in a linear space---we refer to \cite{PeterBook,RBHyp,Benjamin2018model,Welper2017,Metric2019,Nair,Cagniart2019} for proofs related to the slow $m$-width decay of these solution sets and the related numerical experiments. Poor accuracy of a linear approximation motivates us to consider a nonlinear approximation space. We consider a nonlinear approximation space based upon transformed snapshots, an introduction to which is as follows.

\subsection{The transformed snapshot approach}
We introduce a spatial transform $\varphi(\cdot,z,\samZ{i}):\Omega\to\Omega$, and approximate the solution $u_N(\cdot,z)$ in the space
\begin{gather}
\mcal X_m(z):=\Pi(\opn{span}\{u_N(\varphi(\cdot,z,\samZ{i}),\samZ{i})\}_{i\in\Lambda(z)}).\label{def Xm}
\end{gather}
Assuming that $\mcal X_N\subset L^2(\Omega)$, $\Pi$ is a projection operator from $L^2(\Omega)$ to $\mcal X_N$ and will be useful later when we operate a discretization of the operator $\mcal L$ on functions in $\mcal X_m(z)$.
The set $\Lambda(z)\subseteq \{1,\dots,m\}$ is some $z$-dependent index set that selects solution snapshots used to approximate $u_N(\cdot,z)$. A preferred choice for $\Lambda(z)$ is the index of parameter samples that lie in some neighbourhood of the target parameter $z$ \cite{Nair,AbgrallL1}. 

For a detailed discussion on the role of $\varphi$, we refer to \cite{Welper2017,WelperHighRes}, here, we summarize the basic idea. Usually, the exact solution is discontinuous along the parameter domain i.e., the function $u(x,\cdot)$ is discontinuous. This prohibits its accurate approximation in a low-dimensional linear space. Therefore, for all $z\in\mcal Z$, we introduce a $\varphi(x,z,\hat z)$ such that the \textit{transformed solution} $u(\varphi(x,z,\cdot),\cdot)$ is (at least) not discontinuous and sufficiently regular. This allows for a low-dimensional linear reduced approximation of the transformed solution. Indeed, $\mcal X_m(z)$ is a standard linear reduced-basis approximation space for the transformed solution set $\{u(\varphi(\cdot,z,\hat z),\hat z)\hsp :\hsp \hat z\in\mcal Z\}$ with the snapshots collected at $\hat z\in \{\samZ{i}\}_i$. For further clarity on the role of $\varphi$, we refer to the example in \Cref{app: example}. Note that whether such a $\varphi$ always exists for a general hyperbolic PDE is unclear, as yet. In the following, we will make due with an ansatz for $\varphi$ that is accurate for the numerical experiments considered later. 

\subsubsection{The online and the offline phase}
We focus on devising an efficient algorithm to approximate the FOM in $\mcal X_m(z)$. Following is a generic description of the online and the offline phases that we consider.
\begin{enumerate}
\item \textit{Offline phase:} compute the solution snapshots $\{u_N(\cdot,z^{(i)})\}_i$ and the snapshots of the spatial transform $\{\varphi(\cdot,z^{(j)},z^{(i)})\}_{i,j}$. 
\item \textit{Online phase:} perform two steps (a) using the snapshots $\{\varphi(\cdot,z^{(j)},z^{(i)})\}_{i,j}$, approximate $\varphi(\cdot,z,z^{(i)})$ by $\varphi_m(\cdot,z,z^{(i)})$, and (b) approximate $u_N(\cdot,z)$ by \begin{gather}
u_N(\cdot,z)\approx u_m(\cdot,z) \in \mcal X_m(z), \label{def um}
\end{gather}
where, without a change in notation, in $\mcal X_m(z)$, $\varphi$ is replaced by $\varphi_m$.

\end{enumerate}
Note that as compared to a linear ROM (see \cite{PeterBook} for details), the above offline-online stages have a few extra steps, the cost of which should be compensated by the superior accuracy of the nonlinear ROM.

For the online phase, the literature offers three different techniques to compute the spatial transform $\varphi_m$ and the reduced solution $u_m$.
\begin{enumerate}
\item \textit{The data-driven approach} that computes both $\varphi_m$ and $u_m$ via some linear/nonlinear interpolation and discards the underlying PDE \cite{Welper2017}.
\item \textit{The semi-PDE approach} that computes $\varphi_m$ in a data-driven manner but computes $u_m$ using the PDE \cite{Nair,RegisterMOR,mojgani2020}.
\item \textit{The PDE approach} that computes both $\varphi_m$ and $u_m$ using the PDE \cite{MATS,Cagniart2019}. 
\end{enumerate}

We consider the semi-PDE approach. We use a Lagrange polynomial interpolation and residual minimization to compute $\varphi_m$ and $u_m$, respectively. The following reasons motivate our choice. (i) We suspect that by discarding the PDE in the online phase, one might miss out on some physics of the problem, resulting in inaccuracies---similar is the motivation behind physics-based neural networks, which, to enhance the accuracy, impose PDE based constraints on the data-driven nonlinear regression problems \cite{NNPDE}. (ii) As \cite{MATS} suggests, for a purely PDE based approach, one might need the characteristic curves of the PDE to compute $\varphi_m$. A characteristics based approach can result in a highly accurate $\varphi_m$ but its complexity can limit it to one-dimensional problems. 

As for the offline phase, to compute the solution snapshots, we use a first-order finite-volume and an explicit Euler time-stepping scheme. To compute the snapshots of the spatial transform, we assume that $\varphi(\cdot,z,\hat z)$ is a spatial shift function i.e.,
\begin{gather}
\varphi(\cdot,z,\hat z)= \Theta[c(z,\hat z)]\hsp\text{where}\hsp \Theta[c(z,\hat z)](x):= x-c(z,\hat z),\hspB \forall z,\hat z\in\mcal Z.\label{ansatz varphi}
\end{gather}
We compute the spatial shift $c(z,\hat z)\in\mbb R^d$ via L2-minimization \cite{RimTR}. We acknowledge that the above ansatz has limited applicability. As the numerical experiments (and the discussion in \cite{Welper2017} and \Cref{sec: snap sptransf}) indicate, its applicability is limited to problems with a single large discontinuity or to multiple discontinuities moving with the same velocity. To cater to a broader class of problems, we will require a more sophisticated $\varphi$ than that considered above---a few example can be found in \cite{Welper2017,WelperHighRes,Nair,RimOT}. We provide an extension of our work to such a general $\varphi$. However, performing numerical experiments with a different $\varphi$ than above is left as a part of our future work. 

\subsection{Residual minimization} Performing residual minimization in the online phase is (at least) as expensive as computing the FOM \cite{GNAT,Nair}. This is undesirable because despite the availability of an accurate nonlinear approximation space, our ROM will be inefficient. To reduce the computational cost, we perform residual minimization on an online-adaptive reduced mesh, where a reduced mesh refers to a subset of the full mesh. The online-adaptivity allows the reduced mesh to change with the parameter. This accounts for the transport-type behaviour of the residual induced by the nonlinearity of the approximation space $\mcal X_m(z)$. We emphasis that the transport-type behaviour is the reason why the standard non-adaptive reduced mesh techniques (considered in \cite{GNAT,KarenColloc,Astrid}) are ineffective---numerical experiments will provide further elaboration.

An offline/online strategy computes the reduced mesh. The offline phase---using any of the techniques from \cite{Astrid,KarenColloc,GNAT}---collects snapshots of the residual and uses them to compute a reduced mesh. The online phase, capturing the transport in the residual, transports/moves the reduced mesh along the spatial domain. We show that this online phase is efficient on a Cartesian mesh and on an unstructured mesh super-imposed on an auxiliary Cartesian mesh, the so-called structured auxiliary mesh (SAM) \cite{SAM}. Note that here, efficiency refers to the scaling of the computational cost with the size of the reduced mesh and not with the dimension of the FOM.

Even for the simple spatial transform given above in \eqref{ansatz varphi}, to the best of our knowledge, there do not exist hyper-reduction techniques that perform an efficient residual minimization  for the nonlinear reduction of hyperbolic problems. In \cite{KevinAuto}, authors propose an auto-encoder based nonlinear approximation and use residual minimization to compute a ROM, but do not cater to making the residual minimization efficient. In \cite{Nair}, authors use a fixed non-adaptive reduced mesh to speed-up residual minimization. However, due to the transport-type behaviour of the residual, at least for the test cases that we consider, such an approach is inefficient. In \cite{Cagniart2019}, authors propose to speed-up residual minimization via a fast computation of the (L2) inner-products. However, due to a lack of numerical evidence, it is unclear whether this technique results in a significant speed-up.

We note that (similar to \cite{Welper2017,Nair,sPOD}) our method is Eulerian as opposed to the Lagrangian methods considered in \cite{RegisterMOR,Rowley2001,frozenROM}. In the Lagrangian methods, using the spatial transform $\varphi$, one can express the parametrized PDE in the Lagrangian coordinates and reduce the resulting equation using standard techniques. A comparison of the Eulerian methods to the Lagrangian ones is left as a part of our future work. Nonetheless, we point out that expressing the PDE in Lagrangian coordinates introduces space-time gradients of the spatial transform $\varphi$. We speculate that ill-conditioned gradients can lead to stability issues in the Lagrangian methods and that no such stability issues should arise with the Eulerian methods. Further rigorous investigation is needed to examine our speculation. 

Apart from snapshot transformation, other techniques to construct a nonlinear approximation space include the use of auto-encoders \cite{KevinAuto}, the embedding of the solution manifold in a Wasserstein metric space \cite{Metric2019}, the online adaptivity of basis \cite{Benjamin2018model,LaxPairs}, and the method of freezing \cite{Sapsis2018,frozenROM}. We leave an exhaustive comparison to these other techniques as a part of our future work.
 
 \subsection{Organisation}
We have organized the rest of the article as follows. \Cref{sec: FOM} presents our FOM. The discussion is a brief recall of the standard first-order finite-volume methods. \Cref{sec: ROM} presents our ROM. We discuss the sampling of the parameter domain, we concretely define the transformed snapshot based approximation space, we present the residual minimization technique and explain the reason behind its complexity scaling with the dimension of the FOM. \Cref{sec: hyper-reduction} presents a online-adaptive reduced mesh technique to speed-up residual minimization. This technique attempts to make our ROM more efficient than the FOM. \Cref{sec: discussion} provides a rationale behind adapting the reduced mesh, and it extends our technique to structured auxiliary meshes (SAMs) and to spatial transforms different than shifting. \Cref{sec: num exp} present numerical examples showcasing the accuracy of our method.
\section{Full-order model (FOM)}\label{sec: FOM} To compute our FOM, we consider an explicit Euler time-stepping scheme and a first-order finite-volume (FV) spatial discretization. The details are as follows. For a parametrized hyperbolic conservation law, the differential operator $\mcal L$ appearing in \eqref{evo gen} is given as 
\begin{equation}
\mcal L(\cdot,z) = \nabla \cdot f(\cdot,z)\hspB\forall z\in\mcal Z.\label{hyp law}
\end{equation}
The flux-function $f(\cdot,z):\mbb R\to\mbb R^d$ is assumed to be convex and at least twice-differentiable. For all $(x,\mu)\in \mbb R^d\times \mcal P$, the initial conditions read $u(x,t = 0,\mu) = u_0(x,\mu)$. We assume that $u_0(\cdot,\mu)$ is compactly supported. As a result, due to a finite speed of propagation, for any finite final time $T$, the solution $u(\cdot,z)$ is also compactly supported. Therefore, we consider a bounded and connected spatial domain $\Omega\subset\mbb R^d$, which, for all $z\in\mcal Z$, contains the support of $u(\cdot,z)$.  Along the boundary $\pd\Omega$ and for all $z\in\mcal Z$, we prescribe $u(\cdot,z) = 0$. 

A discretization of the space-time domain is as follows. We consider a two-dimensional square spatial domain\footnote{The differences arising from a non-Cartesian discretization of a square domain are outlined later. An extension to a cuboid discretized with a Cartesian mesh is straightforward. An extension to arbitrary curved domains is left as a part of our future work.} i.e., $d=2$ and $\Omega = [0,1]^2$. We consider $N_x\in\mbb N$ number of mesh elements in each spatial direction, resulting in a grid-size of $\Delta x := 1/N_x$. 
To discretize the time-domain $D$, we consider the discrete time-steps $\{t_k\}_{k=0,\dots,K}$ ordered such that 
\begin{gather}
0=t_0 < t_1\dots < t_{K}=T.\label{time steps}
\end{gather}
For simplicity, we consider a constant time-step of $\Delta t$.

An explicit Euler time stepping scheme and a FV spatial discretization of the evolution equation \eqref{evo gen} provides
\begin{gather}
U(t_{k+1},\mu)  = U(t_k,\mu) + \Delta t\times \mcal F(U(t_k,\mu)),\hspB\forall k\in \{0,\dots,K-1\},\mu\in\parSp. \label{FOM}
\end{gather}
The vector $U(t_{k},\mu)\in\mbb R^N$ collects the FV degrees-of-freedom of our FOM. Note that for our FV discretization, $N = N_x^2$. The operator $\mcal F:\mbb R^{N}\to\mbb R^{N}$ is a result of a FV discretization and its explicit expression can be found in any standard textbook (for instance, \cite{LevequeBook}) on FV methods. As the numerical flux (contained in $\mcal F$), we choose the Local-Lax-Friedrich (LLF) flux---a different choice does not change the forthcoming discussion.
We choose the time-step $\Delta t$ to satisfy the CFL-condition
\begin{gather}
\Delta t\leq \frac{\Delta x}{\underset{x\in\Omega,z\in\mcal Z}{\opn{sup}}2|\pd_u f(u,z)\rvert_{u=u_N(x,z)}|}. \label{CFL}
\end{gather}
Under the assumption that for all $\mu\in \mcal P$, $u_0(\cdot,\mu)\in L^\infty(\Omega)$, the CFL-condition ensures that our FOM is $L^{\infty}$ stable \cite{FVNotes}.

\section{Reduced-order model (ROM)}\label{sec: ROM}
We recall and elaborate on the summary of the offline and the online phases presented in the introduction, the related details are discussed later.
\subsection{Summary: Offline phase}
The offline phase consists of the following steps. 
\begin{enumerate}
\item Collect the parameter samples $\{z^{(i)}\}_{i=1,\dots,m}\subset\mcal Z.$
We sample the parameter domain uniformly, the details are given in \Cref{sec: sample par}. 
\item Compute the solution snapshots at the parameter samples. 
\item Compute the snapshots of the spatial transform $\{\varphi(\cdot,z^{(j)},z^{(i)})\}_{i,j}$. As we made the spatial shift ansatz \eqref{ansatz varphi}, we will compute the snapshots of the shifts $\{c(z^{(j)},z^{(i)})\}_{i,j}$---\Cref{sec: snap sptransf} outlines the technique. The technique builds upon the L2-minimization techniques developed earlier in \cite{Welper2017,RimTR,RegisterMOR,Nair}. 
\end{enumerate}

\subsection{Summary: Online phase}
For a target parameter $z=(t,\mu)\not \in \{z^{(i)}\}_i$, the following steps of the online phase compute a reduced approximation $u_m(\cdot,z)$.
\begin{enumerate}
\item \textit{Introduce discrete time-steps:} Consider the discrete time-steps $0=t_0 < \dots <  t_{K^*} = t$. Same as the FOM, consider a constant time-step size of $\Delta t$. For simplicity, we do not explore the possibility of a different time-step for the ROM; see \cite{Cagniart2019} for details. For each $t_k\in \{t_j\}_{j}$, perform the following steps.
\begin{enumerate}
\item[(a)] \textit{Approximate $\{\varphi(\cdot,t_k,\mu,z^{(i)})\}_i$ or the shifts $\{c(t_k,\mu,z^{(i)})\}_i$:} Using the snapshots of the shifts $\{c(z^{(j)},z^{(i)})\}_{i,j}$, approximate $\{c(t_k,\mu,z^{(i)})\}_i$ via
\begin{gather}
c(t_k,\mu,z^{(i)})\approx c_m(t_k,\mu,z^{(i)}),\hspB\forall i\in \{1,\dots,m\}.  \label{def cm}
\end{gather}
To compute $c_m(t_k,\mu,z^{(i)})$, one can use any of the linear/nonlinear regression techniques. 
With a uniform sampling of the parameter domain (given in \Cref{sec: sample par}), we consider a Lagrange polynomial interpolation/regression. Note that the above approximation of the shift implies 
$$
\varphi(\cdot,t_k,\mu,z^{(i)})\approx \varphi_m(\cdot,t_k,\mu,z^{(i)})=\Theta[c_m(t_k,\mu,z^{(i)})],
$$
with $\Theta$ being the spatial shift given in \eqref{ansatz varphi}.
\item[(b)] \textit{Compute the reduced approximation:} find $u_m(\cdot,t_k,\mu)$ in $\mcal X_m(t_k,\mu)$ via residual minimization given as \cite{AbgrallL1,Nair,Cagniart2019}
\begin{equation}
\begin{aligned}
u_m(\cdot,t_k,\mu)=\argmin_{v\in\mcal X_m(t_k,\mu)}&\|v-u_m(\cdot,t_{k-1},\mu)\\
                                                    &+\Delta t\times \mcal L_N(u_m(\cdot,t_{k-1},\mu),t_{k-1},\mu)\|_{L^2(\Omega)}. \label{res min raw}
\end{aligned}
\end{equation}
We initialize with $u(t_0,\mu) = \argmin_{v\in\mcal X_m(t_0,\mu)}\|v-u_0(\cdot,\mu)\|_{L^2(\Omega)}$. 
The approximation space $\mcal X_m(t_k,\mu)$ is given in \eqref{def Xm}. The operator $\mcal L_N$ approximates the evolution operator $\mcal L$. Its precise form is not important here, and we discuss the details in \Cref{sec: res min}. As \Cref{sec: comp cost} clarifies, the complexity of solving the above problem scales with the dimension of the FOM therefore, we later equip it with hyper-reduction. 
\end{enumerate}

\end{enumerate}

\begin{remark}[Computing the projection operator $\Pi$]\label{remark: inclusion}
On a Cartesian mesh, we define $\Pi$ (given in \eqref{def Xm}) by approximating the shifts by an integer multiple of $\Delta x$. This provides 
\begin{equation}
\begin{aligned}
\Pi\left(\opn{span}\{u_N(\right.&\left.\Theta[c_m(t_k,\mu,z^{(i)})],z^{(i)})\}_{i\in\Lambda(t_k,\mu)}\right) \\
:= &\opn{span}\{u_N(\Theta\left[\left\lfloor \frac{c_m(t_k,\mu,z^{(i)})}{\Delta x}\right\rfloor \Delta x\right],z^{(i)})\}_{i\in\Lambda(t_k,\mu)},
\end{aligned}
\end{equation}
where $\lfloor n\rfloor$ represents the greatest integer less than equal to $n$.
One can check that the space on the right is included in $\mcal X_N$. The above definition is faster to compute than an orthogonal projection operator, which requires a numerical quadrature routine. However, on a general unstructured mesh, the above space on the right is not necessarily included in $\mcal X_N$ and we resort to using the orthogonal projection operator. This slightly increase the computation cost of the algorithm, \Cref{sec: extend SAM} discusses the details. 
\end{remark}
\begin{remark}[Boundary conditions for shifting]
For $x\not\in\Omega$ and some shift $c^*\in\mbb R^d$, we define $\Theta[c^*](x) = 0$, which is consistent with the fact that $u(\cdot,z)$ is compactly supported in $\Omega$.
\end{remark}
\subsection{Sampling the parameter domain}\label{sec: sample par}
We consider a two-dimensional parameter domain $\mcal Z$---an extension to higher dimensions is straightforward. We assume that $\mcal Z$ is (or can be mapped via a bijection to) a rectangle. We take $N_t$ and $N_\mu$ uniformly placed samples from $D$ and $\mcal P$, respectively. The vertices of $\mcal Z$ are included in the samples, and the samples from $D$ are a subset of the time-instances $\{t_k\}_{k=0,\dots,K}$ used to compute the FOM. To collect the samples from $\mcal Z$, we take a tensor-product of the samples in $D$ and $\mcal P$. 

\begin{remark}[Scaling with $p$]
Uniform sampling can make snapshot computation unaffordable for large values of $p$. In that case, one can consider a greedy/sparse sampling technique \cite{RBBook,DoubleGreedy}.  
\end{remark}

\subsubsection{Neighbours under uniform sampling}
\Cref{fig: mesh} shows the parameter samples that we identify as the neighbours of a parameter. This provides the index set $\Lambda(z)$ appearing in the approximation space $\mcal X_m(z)$ given in \eqref{def Xm}. A uniform sampling of $\mcal Z$ induces a structured mesh of $(N_t-1)\times (N_\mu-1)$ parameter elements. Each parameter $z\in\mcal Z$ belongs to either a parameter element or to its boundary---generically, we denote the parameter element by $\mcal I^{\mcal Z}$  and collect its definition below. In the first case, the vertices of $\mcal I^{\mcal Z}$ are the neighbours of $z$, whereas in the latter case, the vertices of $\mcal I^{\mcal Z}$ closest to $z$ are its neighbours.


\begin{definition}[Representative element $\mcal I^{\mcal Z}$]\label{def: rep el}
We define $\mcal I^{\mcal Z}$ as $\mcal I^{\mcal Z} := (t^*_1,t^*_2) \times (\mu^*_1,\mu^*_2)$, where $t^*_i$ and $\mu_i^*$ are some elements of $D$ and $\mcal P$, respectively.
The vertices of $\mcal I^{\mcal Z}$ are given---in a counter-clockwise fashion---as $\zRef{1} = (t^*_1,\mu^*_1)$, $\zRef{2} = (t^*_2,\mu^*_1)$, $\zRef{3} = (t^*_2,\mu^*_2)$, and $\zRef{4} = (t^*_1,\mu^*_2)$. 
\end{definition}


\begin{figure}[ht!]
\centering
\begin{tikzpicture}
\draw[->] (-3,0) -- (-2,0);
\draw[->] (-3,0) -- (-3,1);
\node [] at (-3.17,1) {$\mu$};
\node [] at (-2,-0.15) {$t$};
\draw[step=1cm] (-1,-1) grid (2,2);
\foreach \i in {0,...,3}
	\foreach \j in {0,...,1}
	 {\node at (-1 + \i,1 + \j)[circle,fill,black,inner sep=1.5pt]{};
	 \node at (-1 + \i,1 - \j)[circle,fill,black,inner sep=1.5pt]{};}
\foreach \i in {0,...,3}
	\foreach \j in {0,...,2}	 
		{\node at (-1 + \i,1 - \j)[circle,fill,black,inner sep=1.5pt]{};}
\draw[draw=black,fill=gray] (-1,1) rectangle (0,2);	 
\node at (-1,1) [circle,fill,red,inner sep=1.5pt]{};
\node at (-1,2) [circle,fill,red,inner sep=1.5pt]{};
\node at (0,1) [circle,fill,red,inner sep=1.5pt]{};
\node at (0,2) [circle,fill,red,inner sep=1.5pt]{};
\node at (1,0) [circle,fill,blue,inner sep=1.5pt]{};
\node at (1,1) [circle,fill,blue,inner sep=1.5pt]{};
\draw (1,0.5) node[cross=4pt] {};
\end{tikzpicture}
\caption{\textit{Bold dots: uniform samples. Red dots: neighbours for parameters inside the gray region. Blue dots: neighbours for parameters at the crossed-out edge. }} \label{fig: mesh}
\end{figure}
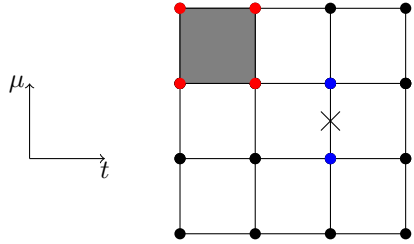

\begin{remark}[The locality of the approximation space $\mcal X_m(z)$]
Consider a $z\in\mcal I^{\mcal Z}$. With our sampling of the parameter domain, the approximation space $\mcal X_m(z)$ in \eqref{def Xm} transforms to
\begin{gather}
\mcal X_m(z) = \Pi(\opn{span}\{u_N(\varphi_m(\cdot,z,\zRef{i}),\zRef{i})\}_{i=1,\dots,4}). \label{def Xm c}
\end{gather}
The space $\mcal X_m(z)$ (and also the one considered in \cite{Nair}) is local---it only uses parameter samples that lie in some neighbourhood of the target parameter. This is consistent with the observation (see \citep{AbgrallL1,Nair}) that, eventually, only the snapshots with parameters that lie in the neighbourhood of $z$ contribute to the approximation of $u_N(\cdot,z)$---the expansion coefficients for all the other snapshots are negligible. 
\end{remark}

\subsection{Residual Minimization}\label{sec: res min} Interpreting the FOM given in \eqref{FOM} as a residual minimization problem provides
\begin{gather}
U(t_{k+1},\mu) = \argmin_{w\in\mbb R^{N}}\|\resFV(w,U(t_{k},\mu))\|^2_{l^2},\hspB\forall \mu\in\mcal P. \label{residual form FV}
\end{gather}
The residual $\resFV:\mbb R^N\times\mbb R^N\to\mbb R^N$ follows from the FOM \eqref{FOM} and is given as
\begin{gather}
\resFV(w,v) := w-v-\Delta t\times \mcal F(v). \label{def residual}
\end{gather}
Obviously, 
$\resFV(w,U(t_{k},\mu)) = 0$ for $w = U(t_{k+1},\mu) $.

We use the above formulation of the FOM to design a time-stepping scheme for our ROM. 
Let $z\in\mcal I^{\mcal Z}$, where $\mcal I^{\mcal Z}$ is the parameter element in \Cref{def: rep el}. The approximation space $\mcal X_m(z)$ is isomorphic to $ \opn{range}(A(z))$. The matrix $A(z)$ is of size $N\times 4$ and contains shifted snapshots as its columns. Equivalently,
\begin{gather}
A(z) := \left(\shiftOp{c_m(z,\zRef{1})}U(\zRef{1}),\dots,\shiftOp{c_m(z,\zRef{4})}U(\zRef{4})\right).\label{def A}
\end{gather}
For some shift $c^*\in\mbb R^d$, the operator $\shiftOp{c^*}:\mbb R^N\to\mbb R^N$ shifts a snapshot along the spatial domain---a precise definition is given below (\Cref{def: shift op}). 

To compute our ROM, instead of minimizing the residual (given in \eqref{residual form FV}) over $\mbb R^N$, we minimize it over the reduced approximation space $\opn{range}(A(t_{k+1},\mu))$. This provides the time-stepping scheme
\begin{equation}
\begin{aligned}
U_m(t_{k+1},\mu)=\argmin_{w\in \opn{range}(A(t_{k+1},\mu))}\|\resFV(w,U_m(t_{k},\mu))\|^2_{l^2},\hspB\forall \mu\in\mcal P. \label{res min temp}
\end{aligned}
\end{equation}
We initialize with $U_m(t_0,\mu):=\argmin_{w\in\opn{range}(A(t_{0},\mu))}\|w-U(t_0,\mu)\|_{l^2}$.

\begin{definition}[Shift operator]\label{def: shift op}
Let $v\in\mbb R^N$ contain the FV degrees-of-freedom of a function $g:\mbb R^d\to\mbb R$ compactly supported inside $\Omega$. Then, with $\mcal T[c^*]v\in\mbb R^N$ we denote a vector that contains the FV degrees-of-freedom of the shifted and projected function $\Pi(g(\Theta[c^*]))$, with $\Pi$ as given in \Cref{remark: inclusion}.
\end{definition}

\subsubsection{Least-squares formulation} Since $U_m(t_{k+1},\mu)\in\opn{range}(A(t_{k+1},\mu))$, we find
\begin{gather}
U_m(t_{k + 1},\mu) =A(t_{k + 1},\mu)\alpha(t_{k + 1},\mu),\label{def alpha}
\end{gather}
where $\alpha(t_{k+1},\mu)\in\mbb R^4$. With the above relation, the minimization problem \eqref{res min temp} transforms to a least-squares problem
\begin{gather}
\alpha(t_{k+1},\mu) = \argmin_{y\in\mbb R^4}\|\underline{A(t_{k+1},\mu)y-b(t_k,\mu)}\|^2_{l^2}, \label{least-square}
\end{gather}
with the vector $b(t_k,\mu)\in\mbb R^N$ given as 
\begin{gather}
b(t_k,\mu) := U_m(t_{k},\mu) + \Delta t\times \mcal F(U_m(t_{k},\mu)). \label{def b}
\end{gather}

Consider the underlined vector in the above problem \eqref{least-square}. The i-th element of this vector represents the residual in the i-th mesh element. Since we take the $l^2$-norm of the entire vector, we minimize the residual over the entire mesh. To highlight this fact, we reformulate the above problem. We collect element ids in the set $\mcal E_z \subseteq \idFull$, where
$$
\idFull := \{1,\dots,N\}. 
$$
The subscript $z\in\mcal Z$ signifies that $\mcal E_z$ can change with the parameter. Furthermore, with $\mcal E_z$, we can identify any subset of the spatial mesh. We refer to this subset as the reduced mesh---\Cref{fig: hyp-red} shows a sample reduced mesh. For simplicity, we assume that the size of $\mcal E_z$ is $z$-independent and denote it by
\begin{gather}
\#\mcal E_z = n.\label{size reduce mesh}
\end{gather}
Obviously, $n \leq N$. 

In the following, we define sub-matrices with rows indexed by the entries in $\mcal E_z$.
\begin{definition}[Sub/Reduced-matrices]\label{def: submat}
Let $V(y)\in\mbb R^{N\times n_2}$, with $y$ belonging to some set $\mcal Y$. With $V[\mcal E_z](y)$ we denote a $n\times n_1$ matrix that contains the $n$-rows of $V(y)$ indexed by $\mcal E_z$ i.e.,
\begin{gather}
\left( V[\mcal E_z](y)\right)_{ij} := \left( V(y)\right)_{ij},\hspB \forall i\in \mcal E_z, j\in \{1,\dots,n_2\}. 
\end{gather}
\end{definition}

With the list $\mcal E_z$ and the above notation, we can choose the mesh elements over which we minimize the residual. The quantity $A[\{i\}](t_{k+1},\mu)y-b[\{i\}](t_k,\mu)$ denotes the residual in the i-th mesh element---same as the i-th entry of the underlined vector in \eqref{least-square}. Thus, to minimize the residual over a reduced mesh identified with a list $\mcal E_z$, we compute
\begin{equation}
\begin{aligned}
\alpha(t_{k+1},\mu) = &\argmin_{y\in\mbb R^4}\sum_{i\in \mcal E_{t_{k+1},\mu}}\left|A[\{i\}](t_{k+1},\mu)y-b[\{i\}](t_k,\mu)\right|^2,\\
=&\argmin_{y\in\mbb R^4}\|A[\mcal E_{t_{k+1},\mu}](t_{k+1},\mu)y-b[\mcal E_{t_{k+1},\mu}](t_k,\mu)\|^2_{l^2}. \label{hyp least-square}
\end{aligned}
\end{equation}
The second equality is nothing but the definition of the $l^2$-norm. \Cref{fig: hyp-red} further elaborates on the above minimization problem.
To minimize the residual over the entire mesh---and recover the earlier least-squares problem \eqref{least-square}---one can choose $\mcal E_{t,\mu} = \idFull$, for all $(t,\mu)\in\mcal Z$.

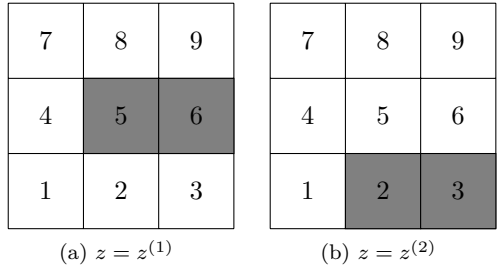
\begin{figure}[ht!]
\centering
\subfloat[$z=z^{(1)}$]{
\begin{tikzpicture}[baseline]
\draw[step=1cm] (-1,-1) grid (2,2);
\draw[draw=black,fill=gray] (0,0) rectangle (1,1);	
\draw[draw=black,fill=gray] (1,0) rectangle (2,1);	
\foreach \i in {1,...,3}
{\node at (-1 + 0.5 + \i - 1,-1 + 0.5) []{\i};}
\foreach \i in {4,...,6}
{\node at (-1 + 0.5 + \i-4,0.5) []{\i};}
\foreach \i in {7,...,9}
{\node at (-1 + 0.5 + \i-7,1.5) []{\i};} 
\end{tikzpicture}}\quad 
\subfloat[$z=z^{(2)}$]{
\begin{tikzpicture}[baseline]
\draw[step=1cm] (-1,-1) grid (2,2);
\draw[draw=black,fill=gray] (0,-1) rectangle (1,0);	
\draw[draw=black,fill=gray] (1,-1) rectangle (2,0);	
\foreach \i in {1,...,3}
{\node at (-1 + 0.5 + \i - 1,-1 + 0.5) []{\i};}
\foreach \i in {4,...,6}
{\node at (-1 + 0.5 + \i-4,0.5) []{\i};}
\foreach \i in {7,...,9}
{\node at (-1 + 0.5 + \i-7,1.5) []{\i};} 
\end{tikzpicture}}
\caption{\textit{A mesh with $N=9$ elements. The reduced mesh is a union of the gray elements. It can change with the parameter, which we refer to as online-adaptivity. With $\mcal E_z = \{1,\dots,9\}$, in \eqref{hyp least-square} we compute the residual on the entire mesh. Whereas, for (a), with $\mcal E_z = \{5,6\}$, we compute the residual only on the reduced mesh.  }} \label{fig: hyp-red}
\end{figure}

\subsection{Computational cost}\label{sec: comp cost}
We study the cost of computing the above least-squares problem. To both compute and minimize the residual, we require $\mcal O(n)$ operations, where $n = \#\mcal E_z$. To minimize the residual, we consider the \texttt{lsqminnorm} routine from \texttt{matlab}, which provides the minimum-norm solution and requires $\mcal O(n)$ operations. Furthermore, to compute the residual, we compute: 
\begin{enumerate}
\item The index set $\mcal E_{t_{k+1},\mu}$, in case it changes with ${t_{k+1},\mu}$. Otherwise---like for the above choice of $\mcal E_z$---we can compute and store $\mcal E_{t_{k+1},\mu}$ offline.
\item The matrix containing the shifted snapshots $A[\mcal E_{t_{k+1},\mu}](t_{k+1},\mu)$.
\item The vector $b[\mcal E_{t_{k+1},\mu}](t_k,\mu)$.
\end{enumerate}
For our online-adaptive technique discussed later (in \Cref{sec: adapt mesh}), the first step requires $\mcal O(n)$ operations. Furthermore, the other two steps also have $\mcal O(n)$ complexity. Two basic operations determine this complexity (i) the shifted snapshot computation, and (ii) the computation of the action of the operator $\mcal F$. We study the complexity of both these operations.

\subsubsection{Cost of computing a shifted snapshot and $\mcal F$} The routine in \Cref{algo: shifted snap} computes the shifted snapshot $\mcal T[c^*]U[\mcal E_z](z)$. Owing to \Cref{remark: inclusion}, we restrict to a shift $c^*$ that is an integer multiple of $\Delta x$. As a result, $\mcal T[c^*]U[\mcal E_z](z)$ is related to $U[\mcal E_z](z)$ via the relation given in \texttt{line-8} and its computation requires the id of the cell that contains the shifted centre. We compute this id via the \texttt{get_location_id} routine given in \Cref{algo: get location id}. The \texttt{get_location_id} routine requires $\mcal O(1)$ operations thus, \Cref{algo: shifted snap} requires $\mcal O(n)$ operations.

To compute $b[\mcal E_{z}](z)$, we require the discreet evolution operator $\mcal F[\mcal E_z](U_m(z))$. To compute this operator, for each cell in $\mcal E_z$, we perform the following operations (i) find the neighbouring cells; (ii) compute the numerical flux at every face; and (iii) integrate the numerical flux along the cell boundary. On a Cartesian mesh, each of these operations have $\mcal O(1)$ complexity. Thus, for the $n$ entries in $\mcal E_z$, we require $\mcal O(n)$ operations to compute $\mcal F[\mcal E_z](U_m(z))$.

\begin{algorithm}[ht!]
\caption{Algorithm to compute a shifted snapshot}
\begin{algorithmic}[1] \label{algo: shifted snap}
\STATE \textbf{Input} $U[\mcal E_z](z)$, $c^*$, $\mcal E_z$, $N_x$, $\Delta x$
\STATE \textbf{Output} $\shiftOp{c^*}U[\mcal E_z](z)$
\STATE $k\leftarrow 1$ 
\FOR{$i\in \mcal E_z$}
\STATE $x_i\leftarrow \texttt{get_cell_centre}(i) - c^*$  \COMMENT{The shifted cell center}
\IF{$x_i\in\Omega$}
\STATE $j\leftarrow\texttt{get_location_id}(x_i,N_x,\Delta x)$ \COMMENT{Values from inside of $\Omega$}
\STATE $\left(\shiftOp{c^*}U[\mcal E_z](z)\right)_k = \left(U[\mcal E_z](z)\right)_{j}$
\ELSE 
\STATE $\left(\shiftOp{c^*}U[\mcal E_z](z)\right)_k = 0$ \COMMENT{Prescribe boundary conditions}
\ENDIF 
\STATE $k\leftarrow k + 1$
\ENDFOR
\end{algorithmic}
\end{algorithm}

\begin{algorithm}[ht!]
\caption{\texttt{get_location_id}: given the location of a space point, computes the id of the mesh element that contains it.}
\begin{algorithmic}[1] \label{algo: get location id}
\STATE \textbf{Input} $x\in\mbb R^2$, $N_x$, $\Delta x$
\STATE \textbf{Output} \texttt{element_id}
\STATE $\texttt{id_x}\leftarrow \texttt{mod}(x(1),\Delta x)$
\STATE $\texttt{id_y}\leftarrow \texttt{mod}(x(2),\Delta x)$
\STATE $\texttt{element_id}\leftarrow (\texttt{id_y}-1)\times N_x + \texttt{id_x}$ 
\end{algorithmic}
\end{algorithm}

\section{Hyper-reduction}\label{sec: hyper-reduction}
With $\mcal E_z = \idFull$, our ROM is (at least) as expensive as the FOM, which is undesirable. While maintaining the accuracy of the ROM, we want to choose $\mcal E_z$ such that $n\ll N$. This way, one can expect the ROM to be more efficient than the FOM. We pursue two approaches to compute such a $\mcal E_z$. 
\begin{enumerate}
\item \textit{The non-adaptive technique} that keeps $\mcal E_z$ fixed in the parameter space.
\item \textit{The online-adaptive technique} where $\mcal E_z$ changes with the parameter i.e., the index set $\mcal E_z$ is online-adaptive. 
\end{enumerate}
First, we present the non-adaptive technique and its shortcomings.

\subsection{Non-adaptive technique}\label{sec: hyp non-adp} This technique consists only of an offline phase. At $\{\bar\mu_j\}_{j=1,\dots,m_{hyp}}\in\mcal P$ parameter samples, we solve the non-hyper-reduced least-squares problem in \eqref{least-square} and collect snapshots of the residuals. This provides the snapshot matrix
\begin{equation}
\begin{aligned}
\mcal S := \left(\right.&\resFV(U_m(t_{2},\bar \mu_{1}),U_m(t_{1},\bar \mu_{1})),\resFV(U_m(t_{3},\bar \mu_{1}),U_m(t_{2},\bar \mu_{1})),\dots,\\
&\left.\resFV(U_m(t_{K},\bar \mu_{m_{hyp}}),U_m(t_{K-1},\bar \mu_{m_{hyp}}))\right).  \label{snap res}
\end{aligned}
\end{equation}

To compute $\mcal E_z$, the non-adaptive technique applies \Cref{algo: compute colloc} (or any other point-selection algorithm from \cite{KarenColloc,GNAT}) to $\mcal S$. The algorithm (taken from \cite{Astrid}) selects the mesh elements with the largest $l^2$-norm of the residual taken over all the parameter samples. Note that as compared to \cite{Astrid}, we apply the algorithm directly to $\mcal S	$ and not to its POD modes. Numerical experiments suggest that both the strategies provide similar results. 

\begin{algorithm}[ht!]
\caption{Summary of the reduced mesh selection algorithm from \cite{Astrid}}
\begin{algorithmic}[1] \label{algo: compute colloc}
\STATE \textbf{Input} $\mcal S$, $n$
\STATE \textbf{Output} $\mcal E_z$
\FOR{$i\in \{1,\dots,N\}$}
\STATE $r(i) = \|\mcal S(i,:)\|_2$ \COMMENT{$\mcal S(i,:)$ denotes the $i$-th row of $\mcal S$. }
\ENDFOR
\STATE $[\texttt{r_sorted},idx] = \texttt{sort}(r)$ \COMMENT{Sort in decreasing order and $\texttt{r_sorted}=r(idx)$.}
\STATE $\mcal E_z\leftarrow idx(1:n)$
\end{algorithmic}
\end{algorithm}

\subsubsection{Shortcomings}
Our numerical experiments (and the example below) suggest that the non-linearity of the approximation space $\mcal X_m(z)$ induces a transport-type behaviour in the residual. As a result, only a large reduced mesh computed using \Cref{algo: compute colloc} can provide a reasonable accuracy---a similar observation holds for the other point-selection techniques outlined in \cite{GNAT,KarenColloc}. This is undesirable because, at least ideally, for some error tolerance of practical interest ($\|u_N(\cdot,z)-u_m(\cdot,z)\|_{L^2}\leq \texttt{TOL}$, for instance), the size of the reduced mesh should be as small as possible. \Cref{algo: compute colloc} provides a parameter-independent reduced mesh, which might not be accurate for the entire range of the parameter domain when the residual shows a transport-type behaviour. Indeed, the motivation to introduce parameter-dependence in the approximation space $\mcal X_m(z)$ was to well-approximate the transport type behaviour of the solution. A parameter-independent linear approximation space (a span of the POD basis, for instance) cannot achieve this task. For further elaboration, consider the example below. The example considers a moving step function and discusses the shortcomings of a non-adaptive reduced mesh.

\begin{example}[Non-adaptive reduced mesh for a moving characteristic function]\label{example: colloc}
Consider $g(\cdot,z) = (1 + z^2)\mathbbm{1}_{[z-0.2,z]}$, where $\mathbbm{1}_{A}$ represents a characteristic function over a set $A\subset\mbb R$. Let $z\in \mcal Z= [0,2]$, and sample $\mcal Z$ at $z^{(1)}=0$ and $z^{(2)}=2$. We approximate $g(\cdot,z)$ by $g_m(\cdot,z)\in span\{g(\Theta[c(z,z^{(i)})],z^{(i)})\}_{i=1,2}$ with $c(z,z^{(i)}) = z-z^{(i)}$. To compute $g_m(\cdot,z)$, we perform a linear interpolation between the transformed snapshots $g(\Theta[c(z,z^{(1)})],z^{(1)})$ and $g(\Theta[c(z,z^{(2)})],z^{(2)})$, which results in the error 
\begin{gather}
E(\cdot,z):=|g(\cdot,z)-g_m(\cdot,z)|=z(2-z)\mathbbm{1}_{[z-0.2,z]}.  \label{error example}
\end{gather}
We interpret $E(\cdot,z)$ as a residual. Note that, similar to $g(\cdot,z)$, the residual also shifts to the right as $z$ increases, exhibiting a transport-type behaviour.

We consider $N = 2\times 10^3$ grid points inside $[-2,2]$, and collect snapshots of $E(\cdot,z)$ at $5$ uniformly placed (excluding the endpoints) parameter samples $\{\bar z^{(i)}\}_{i=1,\dots,5}$ inside $\mcal Z$. This provides the snapshot matrix $\mcal S_{ij}=E(x_i,\bar z^{(j)})$. \Cref{fig: ex colloc} shows these snapshots. We compute a reduced mesh of size $n=100$ using \Cref{algo: compute colloc}.  The locations of the centres of a few of the reduced mesh elements, over-plotted on a few of the snapshots of the error (different from those contained in $\mcal S$), are shown in \Cref{fig: ex colloc}. We make the following observations. (i) Obviously, the reduced mesh does not change with $z$. (ii) For all the three snapshots (and also for the several others not shown in the plot), none of the reduced mesh elements lie inside the support of the residual. For a residual minimization based technique, this could either result in a trivial ROM (i.e., $u_m = 0$) or, as the later numerical experiments indicate, can make the ROM unstable. (iv) At least $800$ mesh elements are needed to have at least one reduced mesh element inside the support of $E(\cdot,z)$ for every $z\in \mcal Z$. Note that $800$ is $40\%$ of the total number of grid points $N$. For such a large reduced mesh, we do not expect to achieve a significant speed-up as compared to the non-hyper-reduced residual minimization.
\begin{figure}[ht!]
\centering
\subfloat[]{\includegraphics[width=2.5in]{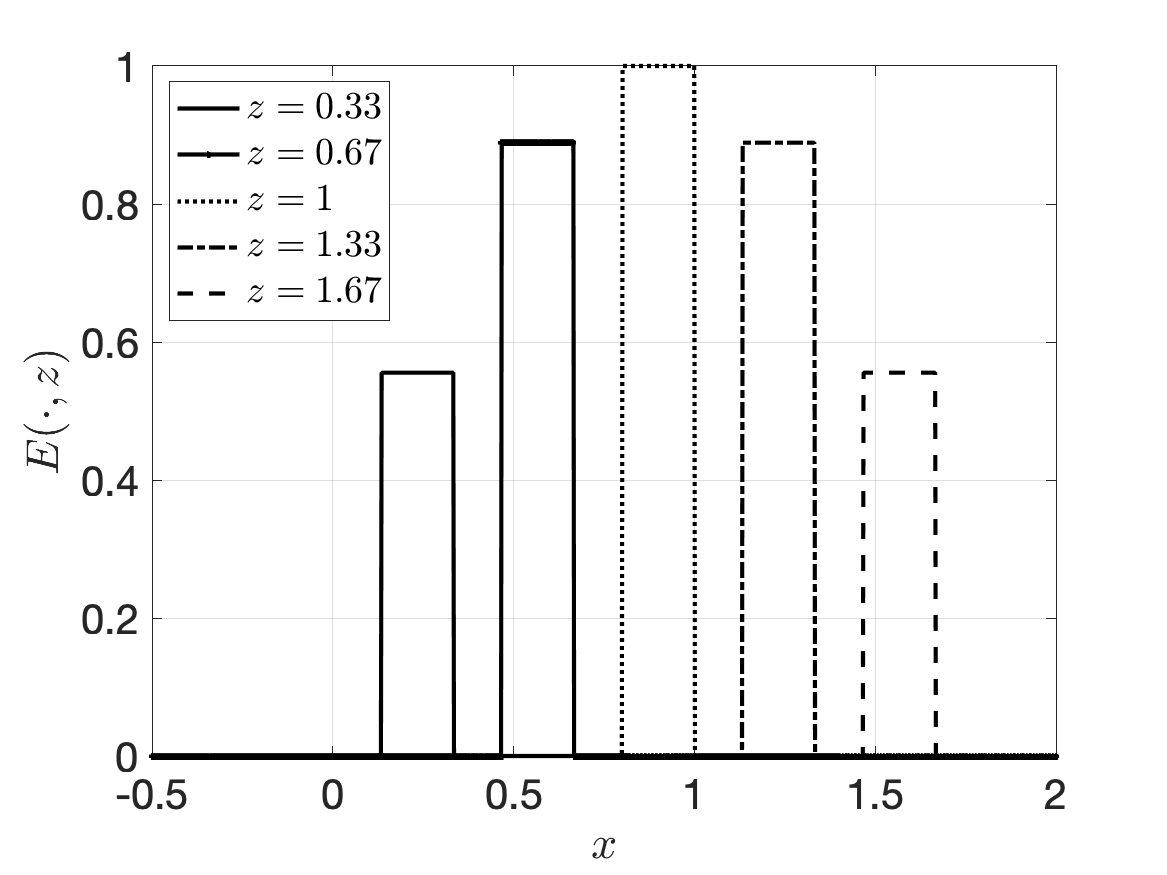}}
\hfill
\subfloat[]{\includegraphics[width=2.5in]{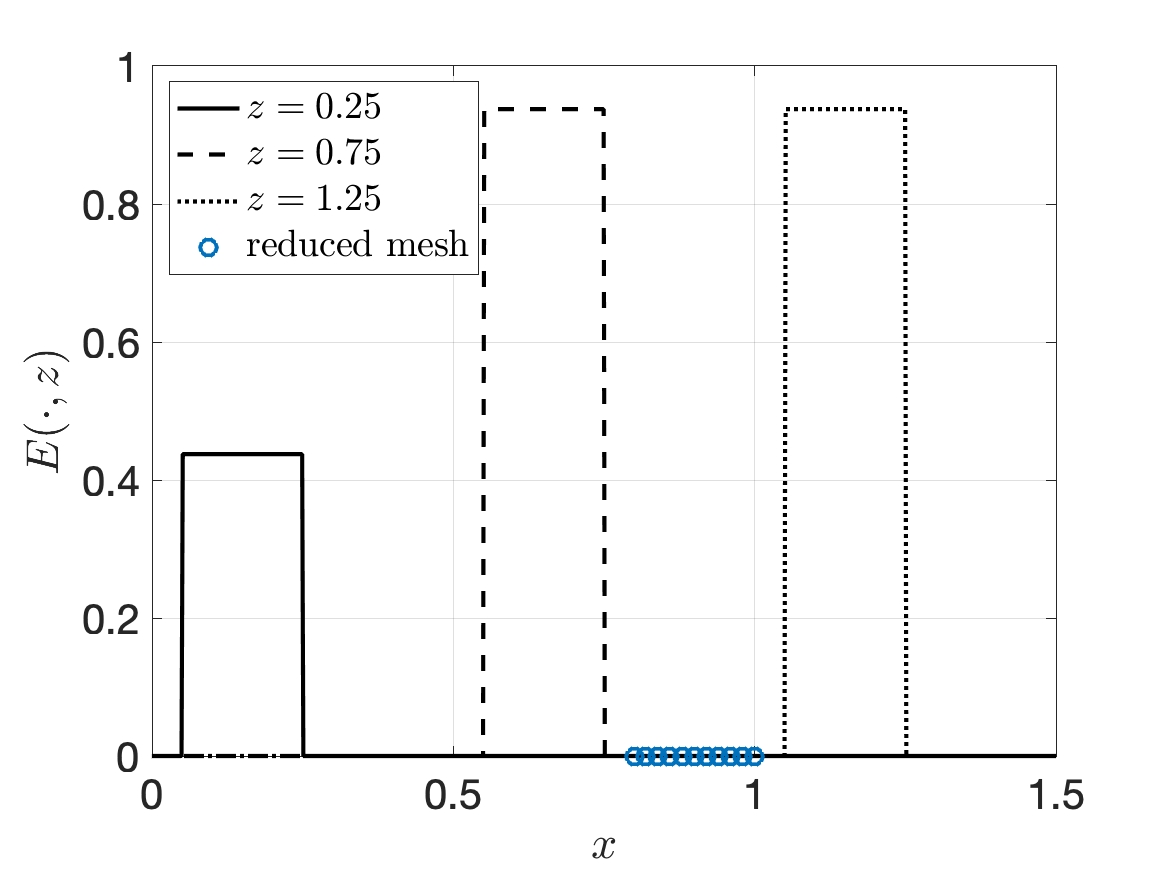}}
\caption{\textit{Results for \Cref{example: colloc}. (a) Snapshots used to compute the reduced mesh. (b) A few snapshots of the residual---different from those used to compute the reduced mesh---and the cell centres of the reduced mesh. The reduced mesh is localized around the support of the snapshot at $z=1$.}}	\label{fig: ex colloc}
\end{figure} 
\end{example}
\subsection{Online-adaptive technique}\label{sec: hyp adp} To account for the transport-type behaviour of the residual, we introduce online-adaptivity in the reduced mesh. This results in an online and an offline stage that computes $\mcal E_{z}$; below is a brief summary. The matrix $\mcal S$ is the snapshot matrix given in \eqref{snap res}. The details of the two stages are given later.
\begin{equation}
\begin{aligned}
\text{Offline Phase}:&\hspB \mcal S\xrightarrow[\text{residuals}]{\text{shift}}\mcal S_{shift}\xrightarrow[\text{reduced mesh}]{\text{select}}\mcal E_{off}\subseteq \{1,\dots,N\}.\\
\text{Online Phase}:&\hspB \mcal E_{off}\xrightarrow[\text{reduced mesh}]{\text{adapt}}\mcal E_{z}.\label{summary hyp}
\end{aligned}
\end{equation}
To select the reduced mesh in the offline phase, we use \Cref{algo: compute colloc}---any other point selection techniques outlined in \cite{KarenColloc,GNAT,Astrid,Nair} will also suffice.
\subsubsection{Offline stage} Similar to the solution, we transform the snapshots of the residuals by evaluating them on a transformed spatial domain---in the language of \cite{RimTR}, this reverses the effect of transport in the residuals. For reasons made clear later in \Cref{sec: shift res}, a transformation based on shifting is sufficient. This provides a snapshot matrix of shifted residuals given as
\begin{equation}
\begin{aligned}
\mcal S_{shift} := \left(\right.&\mcal T[-c_m( t_{2},\bar \mu_1,z_{ref})]\resFV(U_m(t_{2},\bar \mu_{1}),U_m( t_{1},\bar \mu_{1})),\dots,\\
&\left.\mcal T[-c_m(t_{K},\bar \mu_{m_{hyp}},z_{ref})]\resFV(U_m(t_{K},\bar \mu_{m_{hyp}}),U_m(t_{K-1},\bar \mu_{m_{hyp}}))\right).  \label{shift snap res}
\end{aligned}
\end{equation}
Above, $z_{ref}$ could be any of the $N_\mu\times N_t$ parameter samples taken from $\mcal Z$--- refer back to \Cref{sec: sample par} for details. Furthermore, \Cref{algo: shifted snap} provides the shifted residuals, and the shifts $\{c_m(t_{i},\bar \mu_{j},z_{ref})\}_{i,j}$ result from a Lagrange polynomial interpolation over the shift snapshots $\{c(z^{(i)},z^{(j)})\}_{i,j}$.

 After reversing the effect of transport in the residuals, we apply \Cref{algo: compute colloc} to $\mcal S_{shift}$. This results in a reduced mesh with the ids contained in $\mcal E_{off}$. The online phase adapts this reduced mesh. 

\subsubsection{Online stage}\label{sec: adapt mesh} The two steps below constitute the \texttt{adapt_reduced_mesh} routine appearing in the online phase of \eqref{summary hyp}. The rationale behind these two steps is described below. We perform these steps for each of the entries in $\mcal E_{off}$.
\begin{enumerate}
\item Compute the cell center $x_c$ of a cell contained in $\mcal E_{off}$. 
\item Using the \texttt{get_location_id} routine given in \Cref{algo: get location id}, find the cell id containing the shifted point $x_c + c_m(t_{k+1},\mu,z_{ref})$. Include this cell id in $\mcal E_{t_{k+1},\mu}$. 
\end{enumerate}
As noted earlier (in \Cref{sec: comp cost}), the \texttt{get_location_id} routine requires $\mcal O(1)$ operations. Thus, computing the entries of $\mcal E_{t_{k+1},\mu}$ requires $\mcal O(n)$ operations. 

Following is a rationale justifying the above two steps. We express the un-shifted residual in terms of the shifted residual to find
\begin{equation}
\begin{aligned}
\resFV(U_m(t_{k+1},\mu),U_m(t_{k},\mu)) = &\shiftOp{c_m(t_{k+1},\mu,z_{ref})}\\
&\underline{\shiftOp{-c_m(t_{k+1},\mu,z_{ref})}\resFV(U_m(t_{k+1},\mu),U_m(t_{k},\mu))}.
\end{aligned}
\end{equation}
The entries in $\mcal E_{off}$ represent a reduced mesh where the (underlined) shifted residual is the largest in magnitude. Thus, to have a reduced mesh where the un-shifted residual $\resFV(U_m(t_{k+1},\mu),U_m(t_{k},\mu))$ is the largest in magnitude, we shift the centres corresponding to the mesh elements in $\mcal E_{off} $ by $c_m(t_{k+1},\mu,z_{ref})$. 

We revisit \Cref{example: colloc} with our online adaptive technique.
\begin{example}[Revisiting \Cref{example: colloc}]
The numerical parameters remain the same as earlier. In the offline phase, we apply \Cref{algo: compute colloc} to the shifted snapshot matrix $(\mcal S_{shift})_{ij}=E(x_i+\bar z^{(j)},\bar z^{(j)})$. This provides $\mcal E_{off}$. In the online phase, we shift $\mcal E_{off}$ by $z/\Delta x$. The results are shown in \Cref{fig: ex colloc adapt}. The reduced mesh is parameter-dependent and adapts to the residual. As a result, unlike the non-adaptive technique, almost the entire reduced mesh lies inside the support of each residual. In the later numerical experiments, this provides superior accuracy as compared to the non-adaptive technique.

\begin{figure}[ht!]
\centering
\includegraphics[width=2.8in]{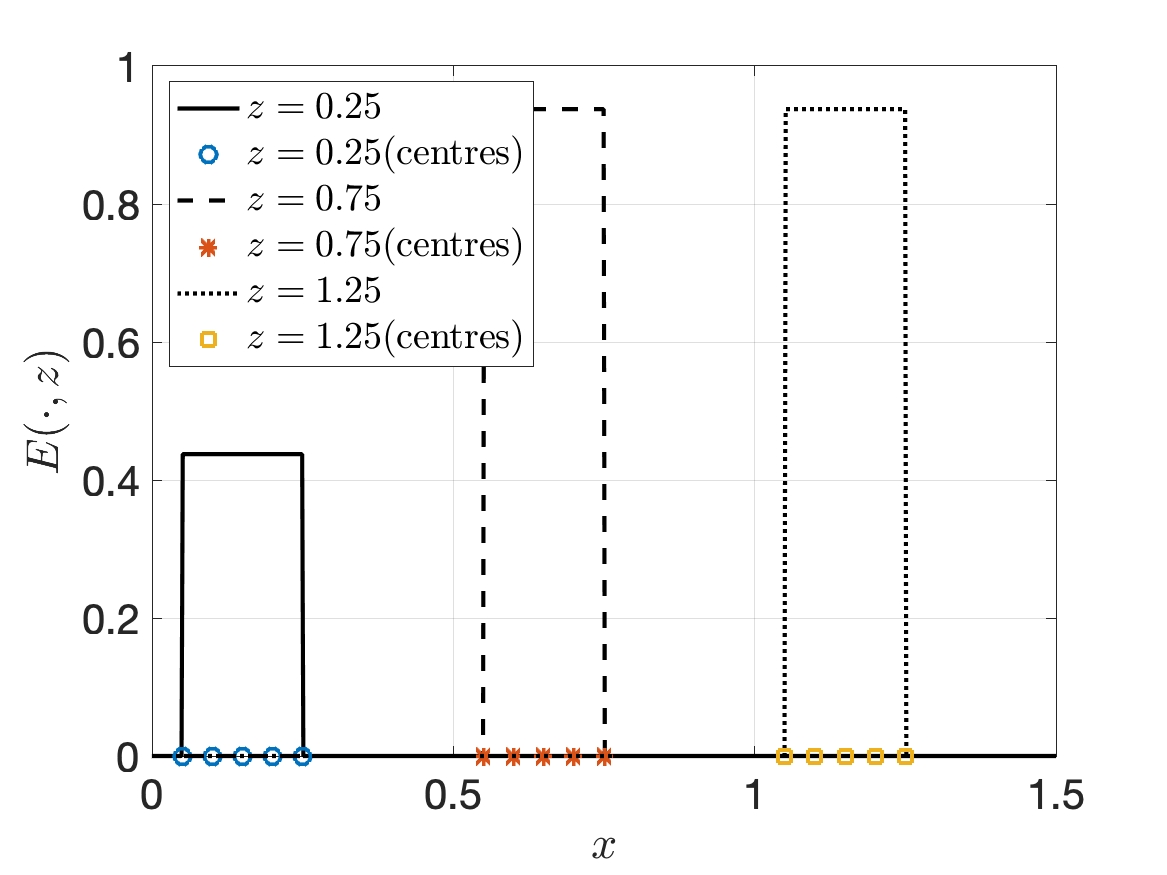}
\caption{\textit{Revisiting \Cref{example: colloc} with the online-adaptive technique. Centres of the online-adapted reduced mesh and the error $E(\cdot,z)$ given in \eqref{error example}.}}	\label{fig: ex colloc adapt}
\end{figure} 
\end{example}

\section{Discussion and extensions}\label{sec: discussion}
\subsection{Why shifting residuals works?}\label{sec: shift res} We explain why a spatial shift can provide a reasonable transformation of the residual or, equivalently, can reverse the effect of transport in the residual. Under some assumptions (listed below in \Cref{lemma: split residual}), we establish that the shifted residual can be given as 
\begin{gather}
\mcal T[-c(t_{k+1},\mu,z_{ref})]\resFV(U_m(t_{k+1},\mu),U_m(t_{k},\mu)) = \resFV^*(t_k,\mu) + \mcal O(\Delta t), \label{break res}
\end{gather} 
where, as $\mu$ and $t_k$ vary over $\mcal P$ and $\{t_j\}_j$, respectively, the function $(t_k,\mu)\mapsto\resFV^*(t_k,\mu)$ has no/minimal transport-type behaviour. The above relation indicates that for a sufficiently small $\Delta t$, shifting should provide an accurate transformation of the residual. 

The following result establishes the above relation, and follows from properties of the shift operator $\mcal T$, the evolution operator $\mcal F$ and some additional assumptions---\Cref{remark: assumptions} further elaborates on the assumptions.
For simplicity and for all $z,\hat z\in\mcal Z$, we assume that the exact shift values $c(z,\hat z)$ are known---similar arguments hold for an accurate approximation $c_m(z,\hat z)$.
\begin{lemma}\label{lemma: split residual}
Assume the following.
\begin{enumerate}
\item[(C1)] $c(z,\hat z)=-c(\hat z,z),\hspB\forall z,\hat z\in\mcal Z$.
\item[(C2)] $c(z,\td z) = c(z,\hat z) + c(\hat z,\td z),\hspB\forall z,\td z,\hat z\in\mcal Z$.
\item[(C3)] $\mcal T[c^*]\mcal F(v) = \mcal F(\mcal T[c^*]v) $. Here, $v\in\mbb R^N$ contains the FV degrees-of-freedom of a function $g:\mbb R^d\to\mbb R$ whose support is contained inside $\Omega$, $c^*$ is a shift such that the support of the shifted function $g(\Theta[c^*])$ is also contained inside $\Omega$, $\mcal F$ is the evolution operator defined in \eqref{FOM}, and $\mcal T[c^*]$ is the shift operator in \Cref{def: shift op}. 
\item[(C4)] $\mcal F\in W^{1,\infty}(\mbb R^N)$. Here, $W^{1,\infty}(\mbb R^N)$ represents a Sobolev space of functions defined over $\mbb R^N$ with bounded weak derivatives upto first-order.
\item[(C5)] Under the CFL-condition \eqref{CFL}, with the initial data in $L^{\infty}$, and with $\mcal E_z=\idFull$, the ROM resulting from the minimization problem \eqref{hyp least-square} is $L^{\infty}$-stable. 
\end{enumerate}
Then, the relation in \eqref{break res} holds with 
\begin{align}
\resFV^*(t_k,\mu) := \td{U}_m(t_{k+1},\mu)-(\opn{Id} + \Delta t\times\mcal F)\td{U}_m(t_{k},\mu),\label{stationary residual}
\end{align}
where $\td {U}_m(z)=\sum_i\alpha_i(z)\mcal T[c(z_{ref},\zRef{i})]U(\zRef{i})$, and $\alpha_i$ is given in \eqref{def alpha}. Furthermore, $z_{ref}$ is the reference parameter appearing in \eqref{snap res}, and $\zRef{i}$ are the vertices of the reference element in \Cref{def: rep el}. 
\end{lemma}
\begin{proof}
See \Cref{app: shift snap}.
\end{proof}

We expect $\resFV^*(t_k,\mu)$ to not have a transport dominated behaviour. First, consider $\td U_m(z)$. It is a linear combination of snapshots taken from the transformed solution set $\{\mcal T[c(z_{ref},z)]U(z)\hsp : \hsp z\in\mcal Z\}$. Assuming that shifting provides an accurate snapshot transformation---otherwise, our ROM will be inaccurate---as $z$ varies over $\mcal Z$, $\td {U}_m(z)$ should not exhibit a dominant transport-type behaviour. 

Now, consider the second term on the right in \eqref{stationary residual} where the operator $\opn{Id} + \Delta t\times \mcal F$ acts on $\td U_m(z)$. As discussed above, $\td U_m(z)$ does not exhibit a transport-type behaviour. Furthermore, due to the CFL-condition \eqref{CFL}, this operator can "move" $U_m(z)$ by a maximum of one grid-cell of size $\Delta x$, which, for a sufficiently small $\Delta t$ and owing to the CFL-condition, is not a dominant transport type behaviour. This justifies our choice to transform the residual with a spatial shift. 

\begin{remark}[Comments on the assumptions made in the above result]\label{remark: assumptions}
(C1) This is an anti-symmetry property and follows from the minimization problem given in \eqref{def c}. (C2) This property states that shifting can be performed in steps, which, at least intuitively, seems reasonable \cite{Welper2017,WelperHighRes}. The intuition (at least for the test cases that we consider) is justified by numerical experiments. (C3) For compactly supported solutions, this property holds true if the continuous-in-space analogue of $\mcal F$, i.e., $\nabla \cdot f$ where $f$ is the flux function, commutes with shifting. This commutation property is a weaker version of Galilean invariance that holds true for most hyperbolic equation of practical relevance---examples include the wave equation, Euler equations, Burgers' equation, etc. (C4) For a twice continuously differentiable flux-function and at least for a LLF numerical flux, one can check that this property holds true. (C5) Our assumption states that the ROM inherits the $L^{\infty}$ stability of the FOM. A rigorous stability proof is unavailable, as yet. Nonetheless, at least for our numerical experiments, the assumption is valid.
\end{remark}

\subsection{Extension to a general spatial transform}\label{sec: extend sptransf} We present an extension of our technique to a general spatial transform $\varphi(\cdot,z,\hat z):\Omega\to\Omega$, which might or might not be the same as the shift function $\Theta[c^*(z,\hat z)]$. Such a transform is suitable for problems involving time-dependent boundary conditions, multiple shocks, shock interaction, etc \cite{Welper2017,Nair,RegisterMOR}. 

 Let $\mcal T[\varphi]$ be the same as the shift operator $\mcal T[c^*]$ given in \Cref{def: shift op} but with the shift function $\Theta[c^*]$ replaced by a general $\varphi$. For most physically relevant PDEs, the evolution operator $\mcal F$ does not need to commute with $\mcal T[\varphi]$ i.e., the assumption (C3) in \Cref{lemma: split residual} might not hold true with $\mcal T[c^*]$ replaced by $\mcal T[\varphi]$. As a result, following the previous discussion, transforming the residual with $\varphi$ might be inaccurate---the transformed residual might still have a dominant transport-type behaviour. Therefore, we need to treat the residual and the solution separately, and find a spatial transform for the residual that is different from that for the solution. This is in contrast to shifting where (thanks to the property (C3)) the same spatial transform, i.e. shifting, is used to transform both the solution and the residual.

For $z,\hat z\in\mcal Z$, let $\varphi_R(\cdot,z,\hat z):\Omega\to\Omega$ represent the spatial transform for the residual---we emphasis that $\varphi_R$ does not need to be the same as $\varphi$. Then, similar to $\mcal S_{shift}$ given in \eqref{shift snap res}, we can define the transformed snapshot matrix of residuals as 
\begin{equation}
\begin{aligned}
\mcal S_{\varphi_R} := \left(\right.&\mcal T[\varphi_R(\cdot,z_{ref},t_{2},\bar \mu_1)]\resFV(U_m(t_{2},\bar \mu_{1}),U_m( t_{1},\bar \mu_{1})),\dots,\\
&\left.\mcal T[\varphi_R(\cdot,z_{ref},t_{K},\bar \mu_{m_{hyp}})]\resFV(U_m(t_{K},\bar \mu_{m_{hyp}}),U_m(t_{K-1},\bar \mu_{m_{hyp}}))\right). 
\end{aligned}
\end{equation}
Offline, we apply \Cref{algo: compute colloc} to $\mcal S_{\varphi_R}$ and compute $\mcal E_{off}$. Online, using the routine \texttt{adapt_reduced_mesh} outlined in \Cref{sec: adapt mesh}, we transform the cell centres of mesh elements whose ids are contained in $\mcal E_{off}$ with the transform $\varphi_R(\cdot,z_{ref},t_{k+1},\mu)$. As stated earlier, on a Cartesian mesh, this procedure requires $\mcal O(n)$ operations. Note that, same as for the solution, one can compute $\varphi_R$ using the non-convex optimization techniques outlined in \cite{Welper2017,RegisterMOR,Nair}.

\subsection{Extension to a Structured Auxiliary Mesh (SAM)}\label{sec: extend SAM}
 We study the computational cost of our ROM technique for a SAM. In a SAM, an unstructured mesh sits atop an auxiliary Cartesian mesh. \Cref{fig: SAM} depicts a simple example, further details can be found in \cite{SAM}. We assume that we have $N_{uc}$ number of unstructured cells in each of the $N_{cc}$ number of Cartesian cells. Thus, the total number of cells read $N = N_{uc}\times N_{cc}$. We assume that $N_{uc}$ is independent of $N$. We claim that on a SAM our ROM requires $\mcal O(N_{uc}N_{quad}n)$ operations, where $N_{quad}$ is the number of quadrature points per cell, and $n$ is the size of the reduced mesh given in \eqref{size reduce mesh}. A justification for our claim is as follows.

The key difference between a SAM and a Cartesian mesh is in the cost of the \texttt{get_location_id} routine. Recall that the routine \texttt{get_location_id} (as given in \Cref{algo: get location id}) finds the id of the mesh element that contains a given space point. On a SAM---first performing a search on the auxiliary Cartesian mesh that requires $\mcal O(1)$ operations---we require an additional $\mcal O(N_{uc})$ operations to search in the $N_{uc}$ unstructured cells contained inside a single Cartesian mesh element. Thus, as opposed to $\mcal O(1)$ operations on a Cartesian mesh, \texttt{get_location_id} routine on a SAM requires $\mcal O(N_{uc})$ operations.

On an unstructured grid, we consider an orthogonal projection operator $\Pi$---refer back to \Cref{remark: inclusion} for details. To compute the projection, we consider a numerical quadrature with $N_{quad}$ quadrature points inside every cell. To find the shifted location of each of these quadrature points, we require $\mcal O(N_{uc}N_{quad})$ operations. To perform this computation for $n$ number of mesh elements, we need $\mcal O(N_{uc}N_{quad}n)$ operations. 

Using the same reasoning as above, one can conclude that, for some $v\in \mbb R^N$, computing the operator $\mcal F(v)$ on a reduced mesh of size $n$ requires $\mcal O(N_{uc}n)$ operations---$\mcal O(N_{uc})$ operations provide the neighbour of a single mesh element. Likewise, adapting a reduced mesh with $n$ elements requires $\mcal O(N_{uc}n)$ operations. Furthermore, the cost of solving the least-squares problem with the \texttt{lsqminnorm} routine requires $\mcal O(n)$. 
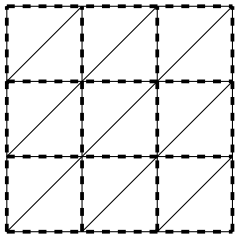
\begin{figure}[ht!]
\centering
\begin{tikzpicture}
\draw[step=1cm] (-1,-1) grid (2,2);
\foreach \i in {0,...,2}
	\foreach \j in {0,...,2}
		{\draw[black] (-1 + \i,1-\j) -- (\i,2-\j);}
\draw[step=1cm,dashed,line width=0.5mm] (-1,-1) grid (2,2);
\end{tikzpicture}
\caption{\textit{A sample SAM. Dashed lines---auxiliary Cartesian mesh. Solid lines---unstructured mesh. }} \label{fig: SAM}
\end{figure}
\section{Numerical Experiments}\label{sec: num exp}
We consider the following test cases. The goal of our numerical experiments is to study the accuracy of the online adaptive reduced mesh technique and compare it to the non-adaptive technique and the FOM.
\begin{enumerate}[label =(\roman*)]
\item \textbf{Test-1 (1D Linear advection) }We consider a linear one-dimensional advection equation with a parameterised advection speed
\begin{linenomath*}
\begin{equation}
\begin{gathered}
\pd_t u(x,t,\mu) + \mu\pd_x u(x,t,\mu) = 0,\hspB \forall (x,t,\mu)\in \Omega\times D\times\mcal P. \label{advection equation}
\end{gathered}
\end{equation}
\end{linenomath*}
We choose $\Omega = [0,3]$, $\parSp= [1,3]$, and $D=[0,0.5]$. The initial data reads
\begin{linenomath*}
\begin{gather}
u_0(x,\mu) = \begin{cases}
\mu,\hspB &x\in [0.5,1]\\
0,\hspB&\text{else}
\end{cases},\hspB\forall \mu\in\mcal P.\label{initial data 1D}
\end{gather}
\end{linenomath*}
\item \textbf{Test-2 (A moving box function)} We construct a reduced approximation to the set $\{u_N(\cdot,z)\hsp :\hsp z\in\mcal Z\}$, where $u_N(\cdot,z)$ is a FV approximation to a function $u(\cdot,z)$ that shifts in $\Omega$ and changes its "shape" with the parameter. For all $(t,\mu)\in D\times\mcal P$, the function $u(\cdot,z)$ is given as
\begin{gather}
u(\cdot,t,\mu) = \begin{cases}
\exp\left(-\mu t\right),\hspB &|x_1-(\mu+t)| \leq 0.3,\hsp |x_2-t| \leq 0.3\\
0,\hspB &\text{else}
\end{cases}.
\end{gather}
We choose $D,\mcal P = [0,1]$, and $\Omega = [-0.5,2.5]^2$. Note that increasing $\mu$ shifts $u(\cdot,t,\mu)$ along the $x_1$-direction, and increasing $t$ shifts $u(\cdot,t,\mu)$ along the vector $(t,t)^T$. To compute the FV approximation $u_N(\cdot,z)$, we project $u(\cdot,z)$ onto the FV approximation space. The details related to the projection and the reduced approximation are discussed later.

\item \textbf{Test-3 (2D Collisionless radiative transport) }We consider the 2D collisionless radiative transport equation given as \cite{Carlos}
\begin{equation}
\begin{aligned}
\pd_t u(x,t,\mu) + &\cos(\mu)\pd_{x_1}u(x,t,\mu)\\
 + &\sin(\mu)\pd_{x_2}u(x,t,\mu) = 0,\hspB \forall (x,t,\mu)\in \Omega\times D\times \mcal P.
 \end{aligned}
\end{equation}
The initial data reads
\begin{gather}
u_0(x,\mu) = \begin{cases}
1,\hspB& \|x\|_2\leq 0.2\\
0,\hspB&\text{else}
\end{cases},\hspB\forall \mu\in\mcal P.
\end{gather}
We set  $\Omega = [-1,1]^2$, $\parSp = [0,2\pi]$, and $D =[0, 0.5]$.
\end{enumerate}
\subsection{Comparison of the different ROMs}
In \Cref{abbrv ROM}, we abbreviate the different ROMs that we compare via numerical experiments. For a $z\in\mcal I^{\mcal Z}$, for the S-ROM, we approximate $u_N(\cdot,z)$ in the space $\td{\mcal X}_m$ given as
\begin{equation}
\begin{aligned}
\td{\mcal X}_m:= \opn{span}\{u_N(\cdot,\zRef{i})\}_{i=1,\dots,4},\label{linear approx}
\end{aligned}
\end{equation}
where $\{\zRef{i}\}_i$ are the parameter samples given in \Cref{def: rep el}. Thus, $\td{\mcal X}_m$ is the same as $\mcal X_m(z)$ but without the spatial transforms. To compute a solution in $\td{\mcal X}_m$, we use the residual minimization technique from \Cref{sec: res min}. We are only interested in the accuracy of the S-ROM and do not equip it with any hyper-reduction technique.

\begin{table}[h!]
\centering
 \begin{tabular}{c | c | c } 
Abbreviation & Approximation space & Hyper-reduction\\
\hline
Adp-SS-ROM & $\mcal X_m(z)$  & online-adaptive (see \Cref{sec: hyp adp})\\
\hline 
N-Adp-SS-ROM &  $\mcal X_m(z)$ & non-adaptive (see \Cref{sec: hyp non-adp})\\
\hline 
SS-ROM & $\mcal X_m(z)$ & none\\
\hline
S-ROM & $\td{\mcal X}_m$ & none\\
 \end{tabular}
 \caption{\textit{Abbreviations for the different ROMs compared via numerical experiments. See \eqref{def Xm c} and \eqref{linear approx} for a definition of $\mcal X_m(z)$ and $\td{\mcal X}_m$, respectively. The abbreviations SS and S stand for shifted snapshot and snapshot, respectively. }} \label{abbrv ROM}
\end{table}

\begin{remark}[Accuracy of the S-ROM]
The approximation space $\td{\mcal X}_m$ satisfies
$$
\td{\mcal X}_m\subseteq \underline{\opn{span}\{u_N(\cdot,z^{(i)})\}_{i=1,\dots,N_\mu\times N_t}},
$$
where the underlined space is a span of all snapshots and is a standard linear reduced basis space for the solution set $\{u_N(\cdot,z)\hsp :\hsp z\in \mcal Z\}$. Since the Kolmogorov $m$-width of this solution set (usually) decays slowly, for a sufficiently small $N_\mu$ and $N_t$, we expect an approximation in the underlined space---and because of the above inclusion, also in $\td{\mcal X}_m$---to be inaccurate. Numerical experiments will corroborate our claim. 
\end{remark}

\subsection{Error quantification} Recall that $N_\mu$ and $N_t$ represent the number of parameter samples along the domains $\mcal P$ and $D$, respectively---see \Cref{sec: sample par} for details. We quantify the error in our ROM via the relative error
\begin{gather}
E(N_t,N_\mu) := \|e\|_{L^\infty(\mcal Z)}\hspB\text{where}\hspB e(z):=\frac{\|u_N(\cdot,z)-u_m(\cdot,z)\|_{L^2(\Omega)}}{\|u_N(\cdot,z)\|_{L^2(\Omega)}}. \label{error ROM}
\end{gather}
The reduced solution $u_m$ can result from either of the ROMs listed in \Cref{abbrv ROM}. We approximate the $L^{\infty}(\mcal Z)$ norm via
\begin{gather}
\|e\|_{L^{\infty}(\mcal Z)} \approx \max_{z\in\mcal Z_{target}}|e(z)|,
\end{gather}
where $\mcal Z_{target}\subset\mcal Z$ is a sufficiently dense, problem dependent and finite set of target parameters given later.
 
\begin{remark}[Software and hardware details]
All the simulations are run using \texttt{matlab}, in serial, and on a computer with two Intel Xeon Silver 4110 processors, 16 cores each and $92$GB of RAM. 
\end{remark}

\subsection{Test-1} We choose a constant time-step of $\Delta t = 1/N_x$, which satisfies the CFL-condition \eqref{CFL}. We choose $N_\mu,N_t=2$ and $N_x=10^3$. Furthermore, as a set of target parameters, we choose
$\mcal Z_{target}=\{(t_i,\td \mu_j)\}_{i,j},$ where $t_j$ are the time-instances at which we compute the ROM, and $\{\td \mu_j\}_j$ are $40$ different parameter samples uniformly placed samples inside $\mcal P$. For the offline phase of the hyper-reduction, we consider five uniformly placed samples inside $\mcal P$ i.e., $m_{hyp}=5$ in \eqref{snap res}. 
\subsubsection{Error comparison}
For the different ROMs outlined in \Cref{abbrv ROM}, \Cref{test-1: err} compares the error $E(N_\mu,N_t)$. The size of the reduced mesh is $n=N_x\times 5\times 10^{-3}$, which is $0.5\%$ of the total mesh size. A few observations are in order. Firstly, with a relative error of $1.07$, the S-ROM performs poorly. It results in an error that is almost five and ten times larger than that resulting from the Adp-SS-ROM and the SS-ROM, respectively. Secondly, the error resulting from the Adp-SS-ROM is twice of that resulting from the SS-ROM. Given the speed-up offered by the Adp-SS-ROM (see the results below), we insist that this loss in accuracy is reasonable. Lastly, the N-Adp-SS-ROM showed large oscillations and appeared to be unstable, which resulted in extremely large error values. Increasing the size of the reduced mesh (as discussed next) makes N-Adp-SS-ROM stable and provides acceptable accuracy. 

The precise reason behind the instability of the N-Adp-SS-ROM is unclear, as yet. Intuitively, we expect that as the residual "moves" along the spatial domain, a fixed reduced mesh is unable to capture the "significant" part of the residual, resulting in instabilities. To further elaborate on this point, for $n=5$, in \Cref{test-1: residual} we plot two snapshots of the residual and the cell centres of the reduced mesh. Similar to \Cref{example: colloc}, for the residual at $z =(0.5,1.5)$, none of the reduced mesh elements lie inside the support of the residual. In contrast, the adaptive approach accurately tracks the "movement" of the residual and appropriately places the reduced mesh.  

\begin{table}[h!]
\centering
 \begin{tabular}{c| c | c | c | c } 
& N-Adp-SS-ROM & Adp-SS-ROM & SS-ROM & S-ROM\\
\hline
$E(N_\mu,N_t)$ & $1.95\times 10^{30}$ & $0.22$ & $0.11$ & $1.07$
 \end{tabular}
 \caption{\textit{Results for test case-1. Error comparison between the different ROMs listed in \Cref{abbrv ROM}. Computations performed with $N_\mu,N_t=2$, $N_x=10^3$, and $n = N_x\times 5\times 10^{-3}$. The N-Adp-SS-ROM showed large oscillations and appeared to be unstable, hence the extremely large error values.}} \label{test-1: err}
\end{table}

\begin{figure}[ht!]
\centering
\includegraphics[width=3in]{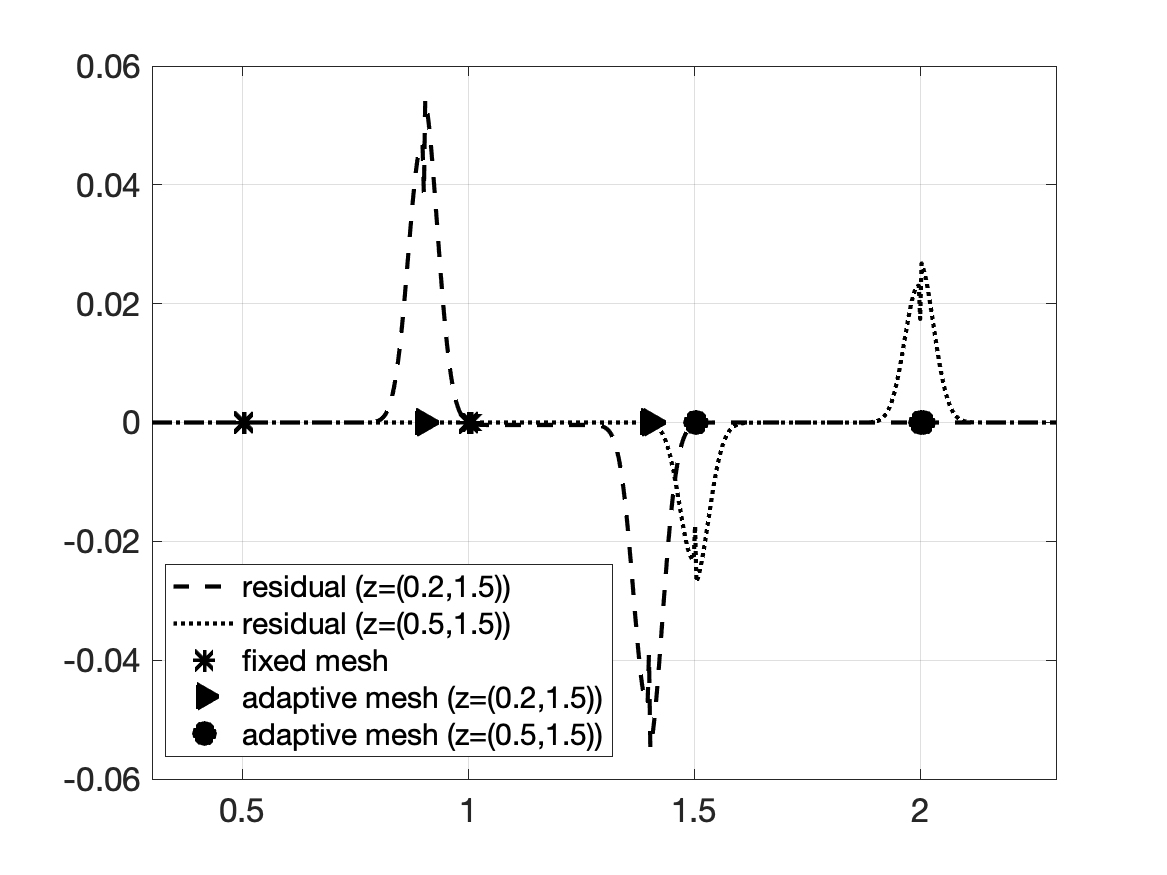} 
\caption{\textit{Results for test case-1. Snapshots of the residual and the centres of the reduced mesh. The size of the reduced mesh is $n=5$. }}\label{test-1: residual}
\end{figure} 

\subsubsection{$E(N_\mu,N_t)$ versus $n$}
For the different sizes of the reduced mesh, we compare the error $E(N_\mu,N_t)$ resulting from the Adp-SS-ROM and the N-Adp-SS-ROM. 
We consider five different values of the reduced mesh size, $n\in\{100,200,400,800\}$. These values correspond to $10\%$, $20\%$, $40\%$ and $80\%$ of the mesh size, respectively. Note that, as the previous study demonstrates, already for $n\ll 100$, the Adp-SS-ROM provides acceptable results. However, for all $n\leq 400$, the N-Adp-SS-ROM showed instabilities, which is the reason behind considering only large values of $n$.

\Cref{test-1: convg n} compares the error values. We make the following observations. (i) Increasing $n$ reduces the error for both the ROMs, which is desirable. Since the accuracy of our ROM is limited by our choice of $N_t$ and $N_\mu$, and increasing $n$ can only offer so much accuracy, the error from Adp-SS-ROM stagnates after a value of $0.11$. (ii) For all $n\leq 400$, the N-Adp-SS-ROM led to strong oscillations and appeared to be unstable, which resulted in extremely large error values. Increasing $n$ to $800$ removes these instabilities and provides reasonable error values. (iii) For all the listed values of $n$, the Adp-SS-ROM remains stable and provides error values close to those reported for the SS-ROM in \Cref{test-1: err}. Note that the Adp-SS-ROM resulted in a larger error value in the previous study because the value of $n$ was smaller. 

\begin{table}[h!]
\centering
 \begin{tabular}{c | c | c } 
 & \multicolumn{2}{c}{$E(N_\mu,N_t)$}\\
 \hline
$n$ $(\% N_x)$ & Adp-SS-ROM & N-Adp-SS-ROM\\
\hline
$100$ (10) & $0.13$ & \underline{$5.6\times 10^{27}$}\\
\hline 
$200$ (20) & $0.11$ & \underline{$1.1\times 10^{26}$}\\
\hline 
$400$ (40)  & $0.11$ & \underline{$1.01\times 10^{13}$}\\
\hline 
$800$ (80) & $0.11$ & $0.11$\\
 \end{tabular}
 \caption{\textit{Results for test case-1. Error versus the size of the reduced mesh $n$. For the underlined values, the N-Adp-SS-ROM showed large oscillations and appeared to be unstable, hence the large error values.  }} \label{test-1: convg n}
\end{table}

\subsubsection{Runtime versus the error} We consider the average runtime given as
\begin{gather}
\mcal C:=\sum_{z\in\mcal Z_{\opn{target}}} \mcal C_{z}/(\#\mcal Z_{\opn{target}}), \label{def runtime}
\end{gather}
where $\mcal C_{z}$ represents the cpu-time (measured with the \texttt{tic-toc} function of \texttt{matlab}) required by the online stage of the ROM (or by the FOM) to compute the solution at the parameter $z$. 

\Cref{test-1: runtime} plots the runtime $\mcal C$ and the speed-up against the error $E(N_\mu,N_t)$. We make the following observations. (i) Although not monotonically, the error converges with $n$. The non-monotonic convergence of the error can be an artefact of the point selection algorithm given in \Cref{algo: compute colloc}. Additional numerical experiments that compare the different point selection algorithms are required to corroborate our claim. (ii) Increasing $n$ increases the runtime. This is consistent with the fact that the cost of the Adp-SS-ROM scales with $n$. (iii) At worst, for $n=320$, the Adp-SS-ROM is $1.8$ times faster than the SS-ROM, and at best, for $n=5$, it is five times faster than the SS-ROM. (iv) The problem is one-dimensional therefore, the explicit time-stepping based FOM is already very efficient. As a result, none of the ROMs provide any speed-up. We refer to test case-3 for a 2D problem where, as compared to the FOM, our hyper-reduction technique offers a significant speed-up. 

\begin{figure}[ht!]
\centering
\subfloat[Error vs runtime]{\includegraphics[width=2.6in]{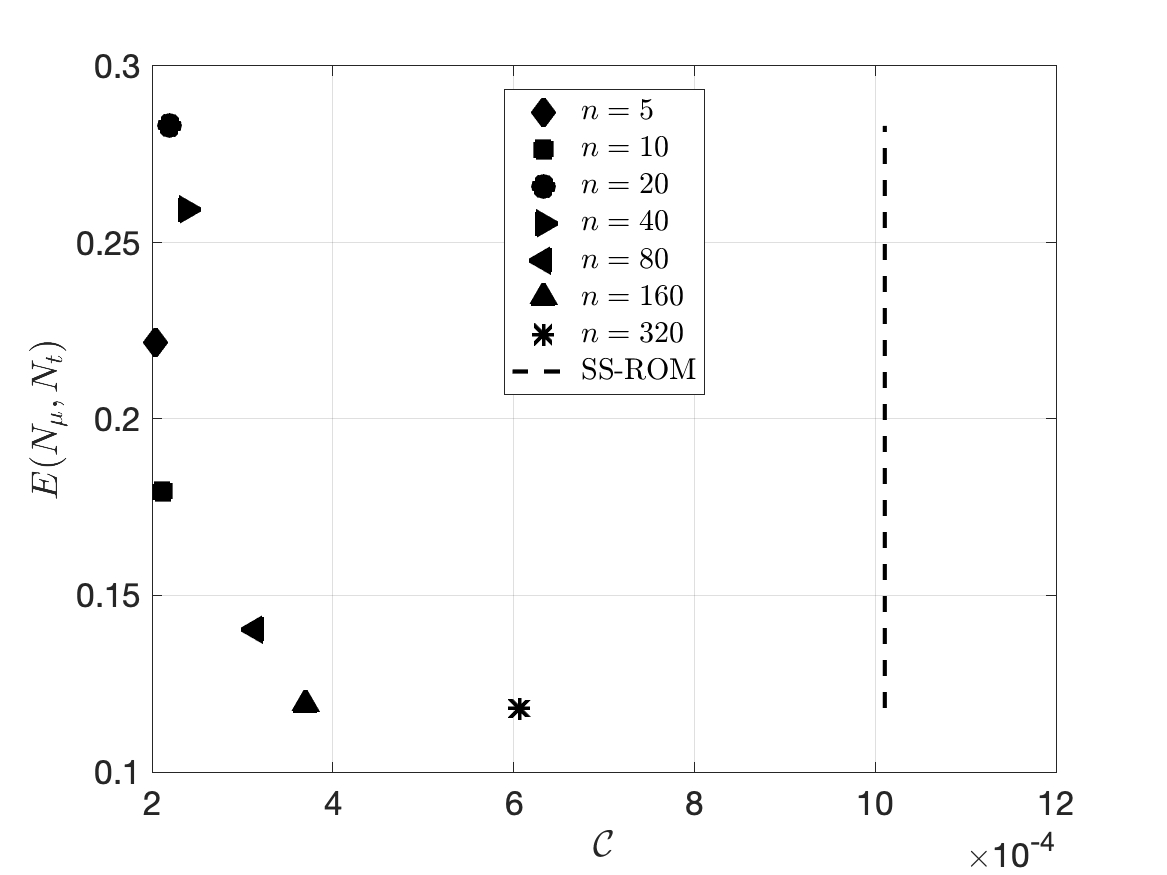}}
\subfloat[Error vs speedup w.r.t the SS-ROM]{\includegraphics[width=2.6in]{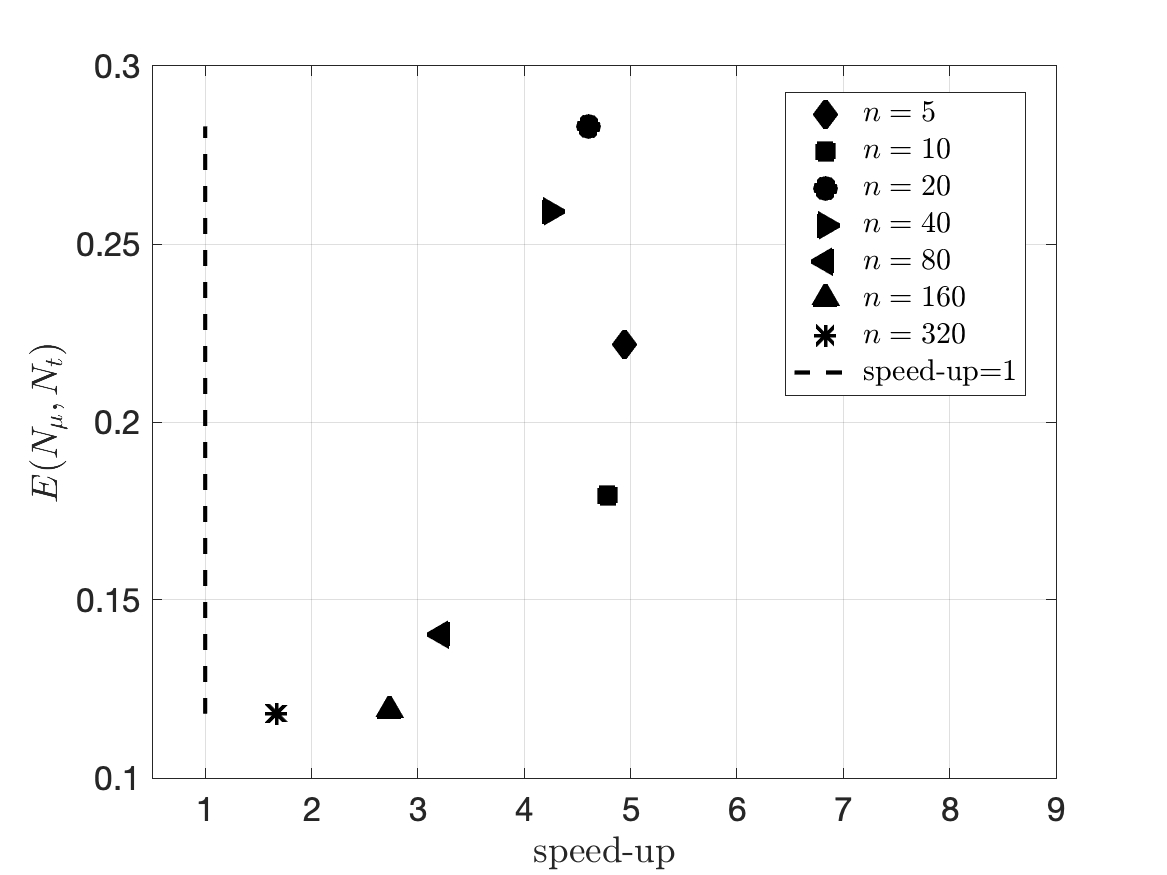}}
\caption{\textit{Results for test case-1 computed with the Adp-SS-ROM. See \eqref{error ROM} and \eqref{def runtime} for a definition of the error $E(N_\mu,N_t)$ and the runtime $\mcal C$, respectively. The dashed line represents (a) the time taken by the SS-ROM, and (b) a speed-up of one.}}\label{test-1: runtime}
\end{figure} 

\subsection{Test-2} We set $N_t=N_\mu = 3$, and $N_x=600$. To compute the SS-ROM, we consider the minimization problem
\begin{gather}
\alpha(z) = \argmin_{y\in\mbb R^4}\|A[\mcal E_z](z)y-b[\mcal E_z](z)\|_{l^2}.\label{ls test2}
\end{gather}
Here, $\mcal E_z$ and $A[\mcal E_z](z)$ are as defined earlier, and $b[\mcal E_z](z)$ is a sub-vector of the vector $b(z)=U(z)$. To compute $U(z)$, we project $u(\cdot,z)$ onto the FV approximation space. To perform the projection, we consider tensorized $5\times 5$ Gauss-Legendre quadrature points in each mesh element. The reduced approximation is given by $U_m =A(z)\alpha(z) $. Note that for the current test case, computing the FOM is equivalent to projecting the exact solution onto the FV approximation space. To collect snapshots of the residual, in \eqref{snap res}, we set $m_{hyp} = 4$. We choose $\mcal Z_{\opn{target}}$ as $100\times 100$ uniformly placed and tensorised points inside $\mcal Z$.

\begin{remark}[No time-stepping]\label{remark: no time-step}
The above minimization problem does not involve a time-stepping scheme. This allows us to study the errors resulting from the reduced approximation, residual minimization and hyper-reduction without the errors introduced from the time-stepping scheme.
\end{remark}

\subsubsection{Error comparison} \Cref{test-2: err} presents the error values resulting from the different ROMs listed in \Cref{abbrv ROM}. As the size of the reduced mesh, we choose $n=N_x^2\times 10^{-2}$, which is $1\%$ of the total mesh size. Both the Adp-SS-ROM and the SS-ROM outperform the S-ROM. The error values resulting from the S-ROM are almost $4.5$ times of those resulting from the Adp-SS-ROM. At least for the present test case and our choice of $n$, our online adaptive hyper-reduction technique introduces almost no error in the SS-ROM.

Unlike the previous test case, the N-Adp-SS-ROM did not exhibit large oscillations, instabilities or extremely large error values. The reason being that the minimization problem in \eqref{ls test2} does not involve a time-stepping scheme---see \Cref{remark: no time-step} above. This prevents error accumulation over time, which, along with a poor placement of the reduced mesh, was one of the reasons why the N-Adp-SS-ROM was unstable in the previous study. For the N-Adp-SS-ROM, there exist target parameters without a single reduced mesh element lying inside the support of the residual. As a result, the solution to the minimization problem \eqref{ls test2} is zero, leading to a relative error of one.

\begin{table}[h!]
\centering
 \begin{tabular}{c| c | c | c | c } 
& N-Adp-SS-ROM & Adp-SS-ROM & SS-ROM & S-ROM\\
\hline
$E(N_\mu,N_t)$ & $1$ & $0.19$ & $0.18$ & $0.84$
 \end{tabular}
 \caption{\textit{Results for test case-2. Computations performed with $N_\mu,N_t=3$, $N_x=600$, and $n = N_x^2\times 10^{-2}$. Error comparison between the different ROMs listed in \Cref{abbrv ROM}. }} \label{test-2: err}
\end{table}
\subsubsection{$E(N_\mu,N_t)$ versus $n$}
For different values of $n$, \Cref{test-2: cong n} compares the error $E(N_\mu,N_t)$ resulting from the N-Adp-SS-ROM and the Adp-SS-ROM. Our observations remain similar to the previous test case. The N-Adp-SS-ROM requires a large reduced mesh to achieve an acceptable accuracy---at least $n=5.76\times 10^4$ reduced mesh elements, which is $16\%$ of the total mesh size, are required to achieve an error of $0.19$. In contrast, the Adp-SS-ROM provides similar accuracy with a reduced mesh that is just $0.5\%$ of the total mesh size. As before, our choice of $N_\mu$ and $N_t$ limit the accuracy, resulting in error stagnation.

\begin{figure}[ht!]
\centering
\includegraphics[width=2.8in]{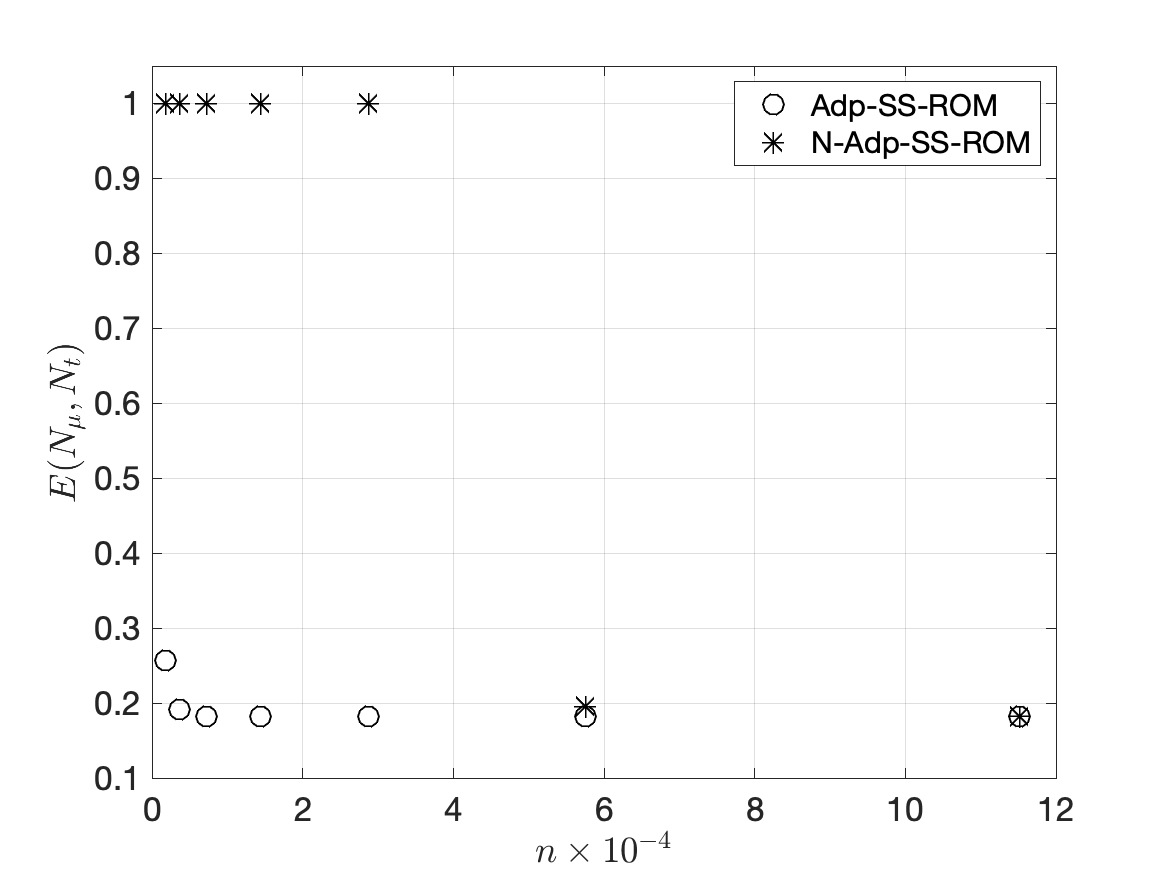}
\caption{\textit{Results for test case-2. Error versus the size of the reduced mesh $n$. Computations performed with $N_t=N_\mu = 3$.}}\label{test-2: cong n}
\end{figure} 

\subsubsection{Runtime versus error} Consider the average runtime $\mcal C$ defined in \eqref{def runtime}. For the Adp-SS-ROM, \Cref{test-2: runtime} plots the runtime and the speed-up against the error. A few observations follow. (i) The error decreases monotonically upon increasing $n$, which is desirable. (ii) At best, for $n=1.8\times 10^3$, the Adp-SS-ROM is $30$ times faster and $1.3$ times worse in accuracy than the SS-ROM. (iii) At worst, for $n=115.2\times 10^3$, the Adp-SS-ROM is almost $8.5$ times faster (and similar in accuracy) than the SS-ROM. (iv) Computing the FOM involves projecting a function onto the FV approximation space, which is a cheap operation. Therefore, none of the ROMs offer any speed-up. We refer to the following test case that considers a more realistic scenario and presents the speed-up offered by our hyper-reduction technique. 

\begin{figure}[ht!]
\centering
\subfloat[Error versus runtime]{\includegraphics[width=2.3in]{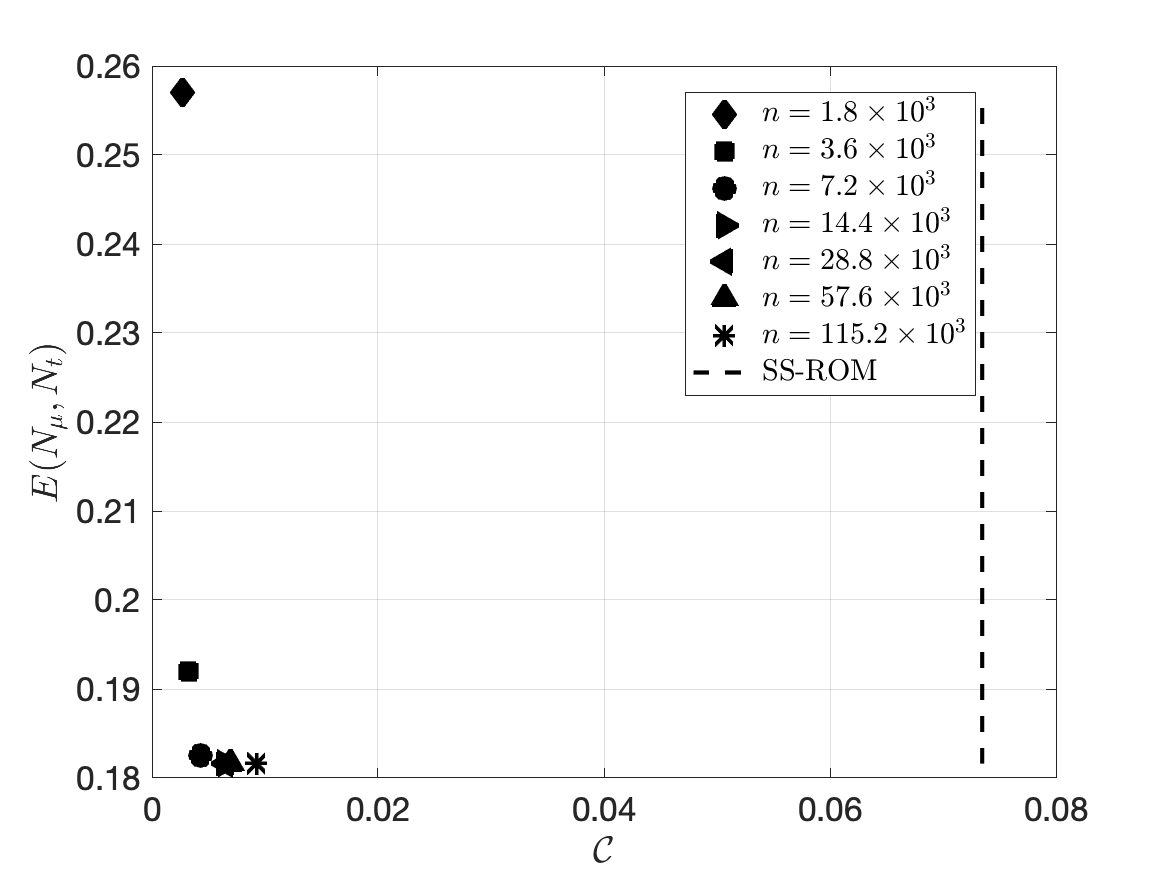}}
\hfill
\subfloat[Error versus speed-up w.r.t the SS-ROM]{\includegraphics[width=2.3in]{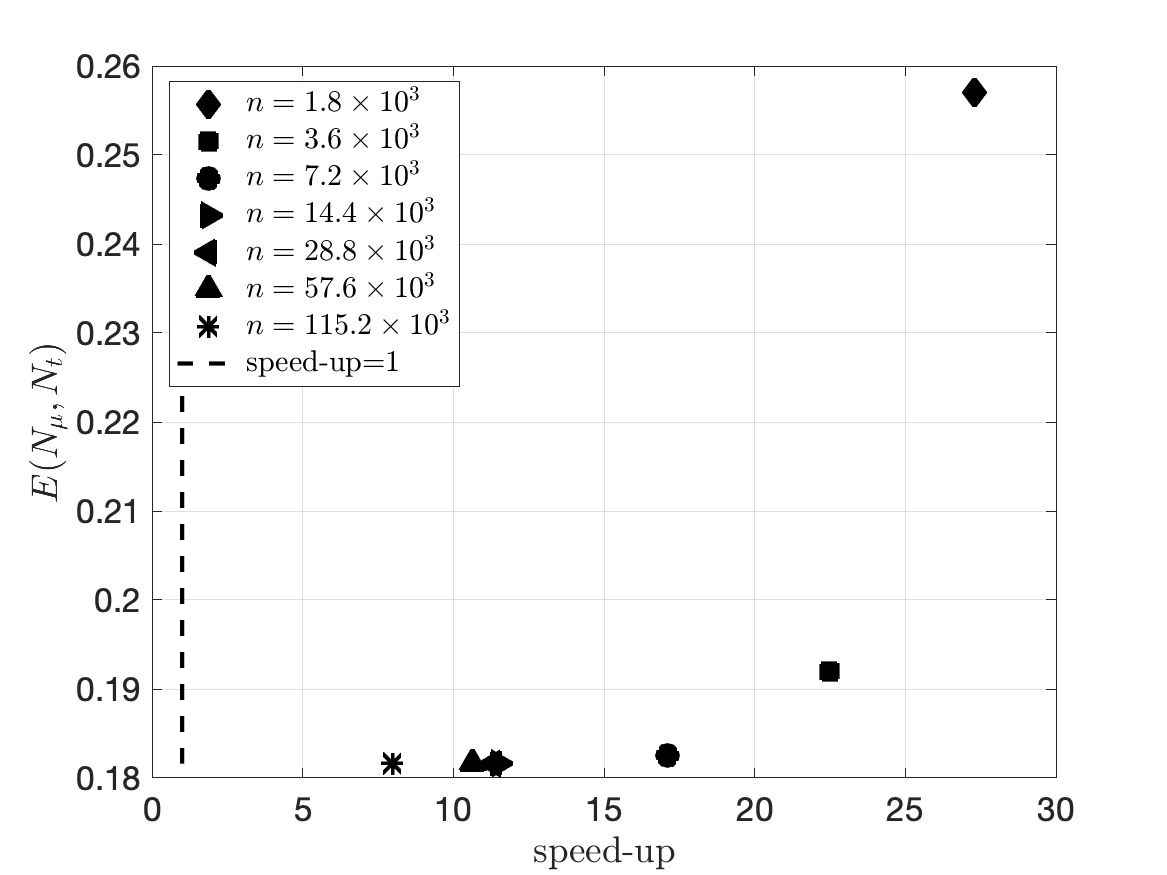}}
\caption{\textit{Results for test case-2, computed with the Adp-SS-ROM. The dashed line represents (a) the time taken by the SS-ROM, and (b) a speed-up of one.}}\label{test-2: runtime}
\end{figure} 

\subsection{Test-3 } We discretize $\Omega$ with a $N_x\times N_x$ Cartesian grid and we choose $N_x=800$. We use a constant time step of $\Delta t = \Delta x/2$. We set $N_\mu=N_t = 6$. To collect snapshots of the residual, we take $5$ uniformly placed samples from $\mcal P$ i.e., $m_{hyp}=5$ in \eqref{snap res}.  We study the ROM at the target parameters $\mcal Z_{target}= \{(t_i,\td \mu_j)\}_{i,j}.$
 Here, $\{t_j\}_j$ represent the time-instances at which we compute the ROM, and $\{\td\mu_i\}_i$ are $50$ uniformly placed samples inside $\mcal P$.

\subsubsection{Error comparison} We choose a reduced mesh that contains $2\%$ of the total mesh elements. \Cref{test-3: err} shows the error $E(N_\mu,N_t)$ for the different ROMs. We make the following observations. (i) Both the Adp-SS-ROM and the SS-ROM outperform the S-ROM. (ii) The maximum error resulting from the Adp-SS-ROM is almost $1.3$ times of that resulting from the SS-ROM. Given that Adp-SS-ROM is 50 times more efficient than the SS-ROM---see the discussion below---we insist that the loss in accuracy introduced via hyper-reduction is acceptable. (iii) The N-Adp-SS-ROM shows large oscillations resulting in large error values. Nonetheless, same as earlier, increasing the size of the reduced mesh removes these instabilities and provides an acceptable accuracy.

\begin{table}[h!]
\centering
 \begin{tabular}{c| c | c | c | c } 
& N-Adp-SS-ROM & Adp-SS-ROM & SS-ROM & S-ROM\\
\hline
$E(N_\mu,N_t)$ & $18.75\times 10^{3}$ & $0.29$ & $0.21$ & $1.06$
 \end{tabular}
 \caption{\textit{Results for test case-3. Computations performed with $N_\mu,N_t=6$, $N_x=800$, and $n = N_x\times 2\times 10^{-2}$. Error comparison between the different ROMs listed in \Cref{abbrv ROM}. The N-Adp-SS-ROM showed large oscillations and appeared to be unstable, hence the extremely large error values.}} \label{test-3: err}
\end{table}

\subsubsection{Runtime versus the error}
For the Adp-SS-ROM, \Cref{test-3: runtime vs err} plots the average runtime and the speed-up versus the error $E(N_\mu,N_t)$. We make the following observations. (i) Increasing $n$ increases the runtime and decreases the speed-up, which is as expected. Beyond $n=25.6\times 10^3$, as compared to the FOM, the Adp-SS-ROM does not offer any speed-up. (ii) The lowest runtime and the maximum speed-up of $7.8$ corresponds to a reduced mesh that contains $0.5\%$ of the total mesh elements. The relative error is $0.32$, which is one-third of that resulting from the S-ROM and is $1.5$ times of that resulting from the SS-ROM---see \Cref{test-3: err}. The accuracy loss as compared to the SS-ROM is acceptable given that the Adp-SS-ROM offers a speed-up of two orders-of-magnitude.

\begin{figure}[ht!]
\centering
\subfloat [Error vs runtime]{
\includegraphics[width=2.3in]{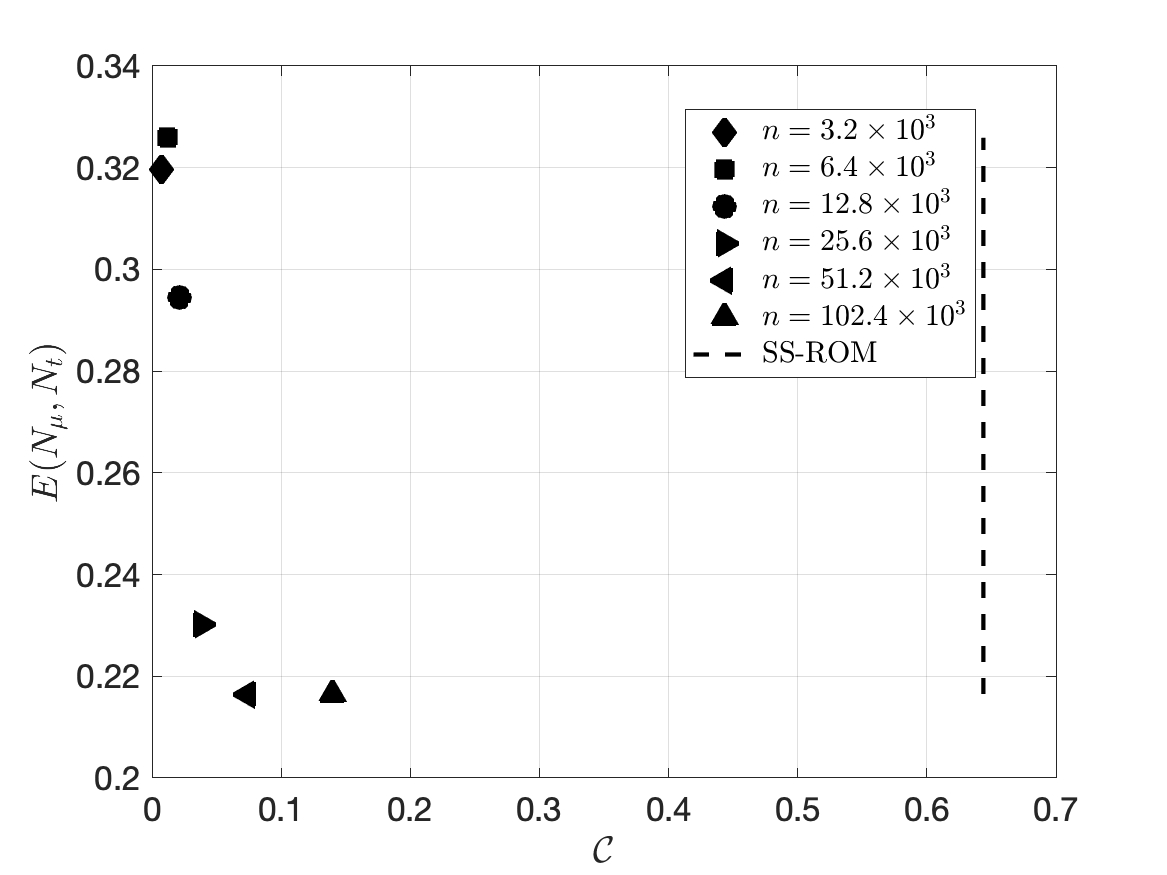} 
}
\hfill
\subfloat [Error vs speed-up w.r.t to the FOM]{
\includegraphics[width=2.3in]{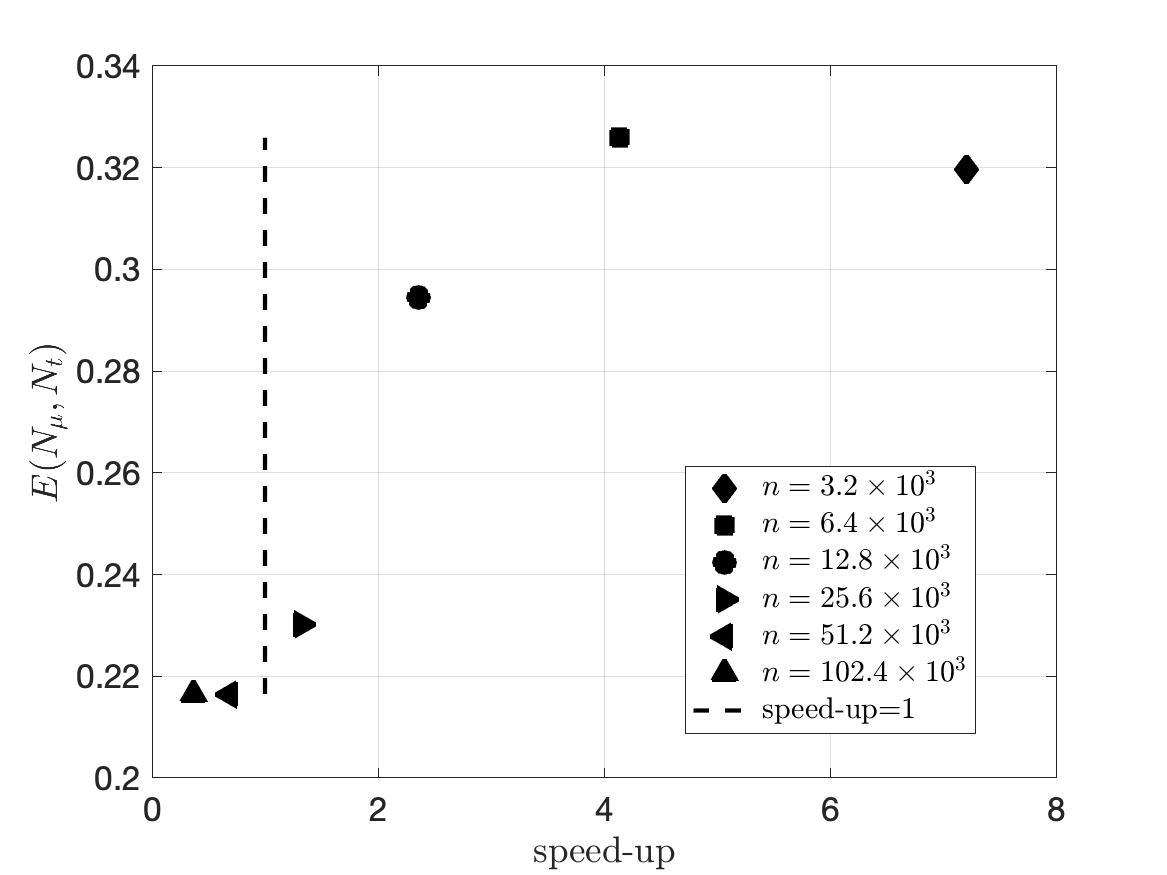} 
}
\hfill
\subfloat [Error vs speed-up w.r.t the SS-ROM]{
\includegraphics[width=2.3in]{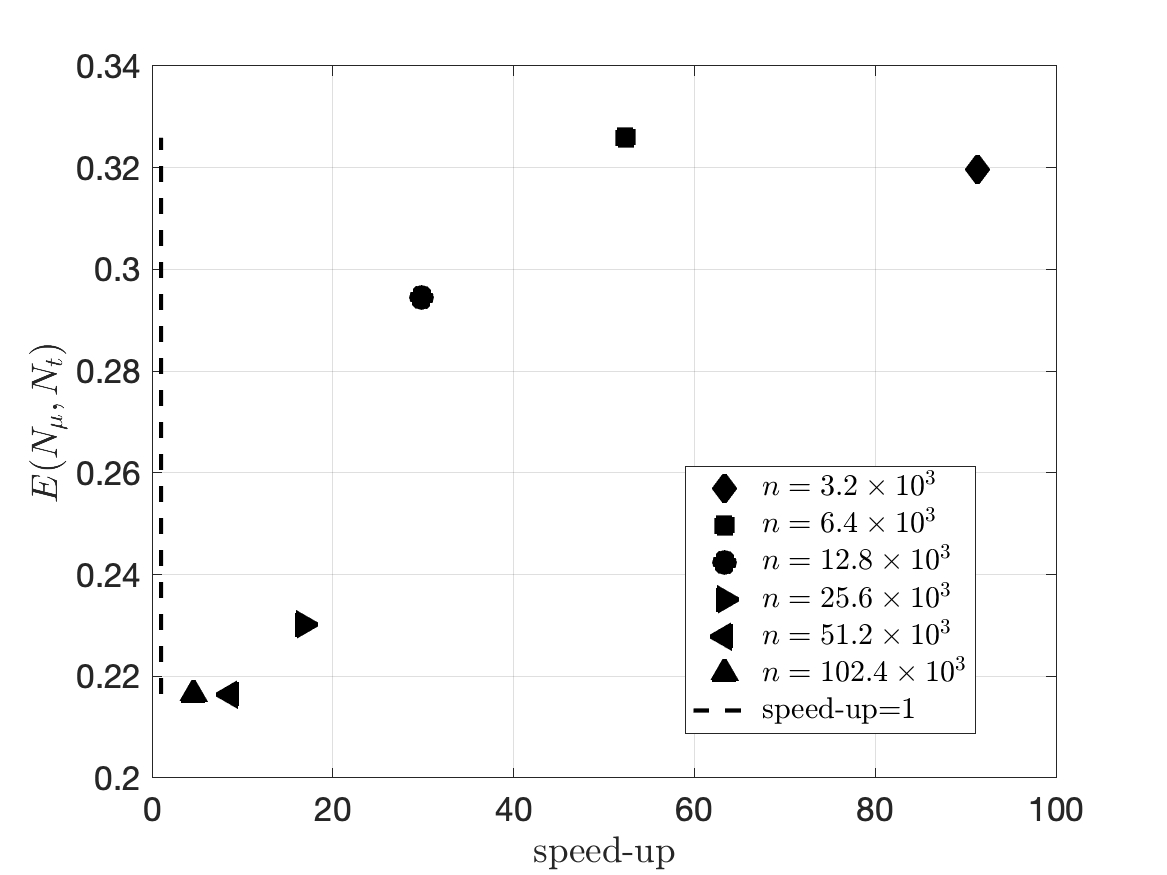} 
}
\caption{\textit{Results for test case-3, computed with the Adp-SS-ROM. Computations performed with $N_\mu = N_t = 6$, and $N_x = 800$. See \eqref{error ROM} and \eqref{def runtime} for a definition of $\mcal E(N_\mu,N_t)$ and $\mcal C$, respectively. The dotted line indicates (a) the runtime of the SS-ROM, (b) and (c) speed-up of one. }}\label{test-3: runtime vs err}
\end{figure} 

\subsubsection{Runtime split}
We split the runtime $\mcal C$ into four major parts. (i) $\mcal C_{adapt}$, the average runtime to adapt the reduced mesh, (ii) $\mcal C_A$, the average runtime to compute the matrix $A(t_{k+1},\mu)$ given in \eqref{def A} (or $A[\mcal E_{t_{k+1},\mu}(t_{k+1},\mu)]$ in the case of hyper-reduction). (iii) $\mcal C_{b}$, the average runtime to compute the vector $b(t_k,\mu)$ given in \eqref{def b}. (iv) $\mcal C_{ls}$, the average runtime to solve the least-squares problem in \eqref{least-square}. For $N_x=800$ and $n=N_x^2\times 2\times 10^{-2}$, \Cref{test-3: split run} compares the different runtime for the Adp-SS-ROM. By far, computing the vector $b(t_k,\mu)$ is the most expensive part of the algorithm---it takes almost $70\%$ of the total runtime. It is noteworthy that the combined cost of solving the least-squares problem and adapting the reduced mesh is less than $10\%$ of the total runtime. Although not shown in the plot, increasing $N_x$ has almost no effect on the runtime. 

\begin{figure}[ht!]
\centering
\includegraphics[width=2.3in]{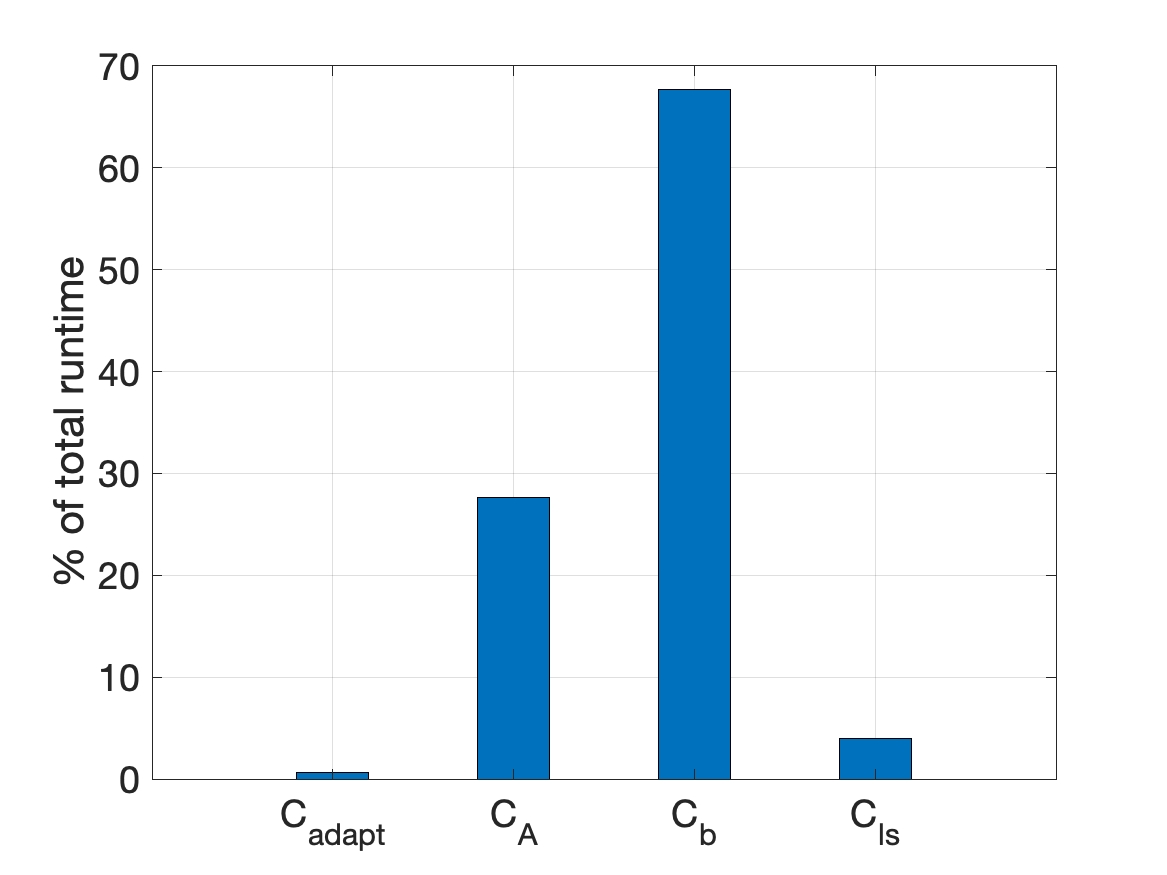} 
\caption{\textit{Results for test case-3. Runtime split for the online-adaptive hyper-reduced ROM. Computations performed with $N_t = N_\mu = 6$.}}	\label{test-3: split run}
\end{figure} 

\subsubsection{Visualization of the reduced mesh}
For $N_x=500$ and $n = N_x^2\times 5\times 10^{-3}$, \Cref{test-3: reduced mesh} presents the centres of the reduced mesh computed using the adaptive and the non-adaptive technique. The reduced mesh resulting from the non-adaptive technique is fixed in the parameter space. Most of its elements are centred around the origin and none lie in the "significant/non-zero" part of the residual. Apparently, this---as noted earlier---results in an unstable ROM. In contrast, by tracking the "movement" of the residual, the adaptive technique changes the reduced mesh with the parameter and places it where the residual is large/non-zero.

\begin{figure}[ht!]
\centering
\subfloat [Cell centres of the non-adaptive reduced mesh.]{
\includegraphics[width=2.3in]{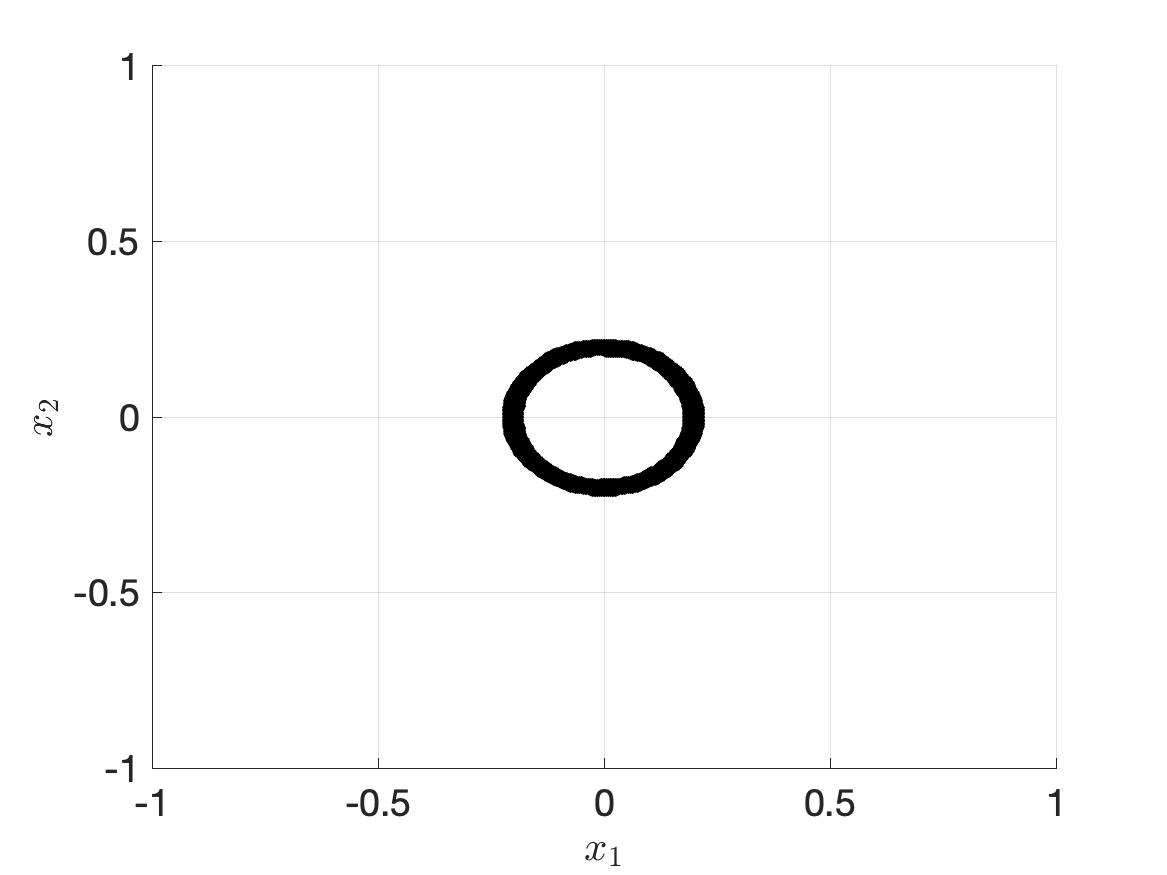} 
}
\hfill
\subfloat [Cell centres of the adaptive reduced mesh for $z = (0.5,4.35)$.]{
\includegraphics[width=2.3in]{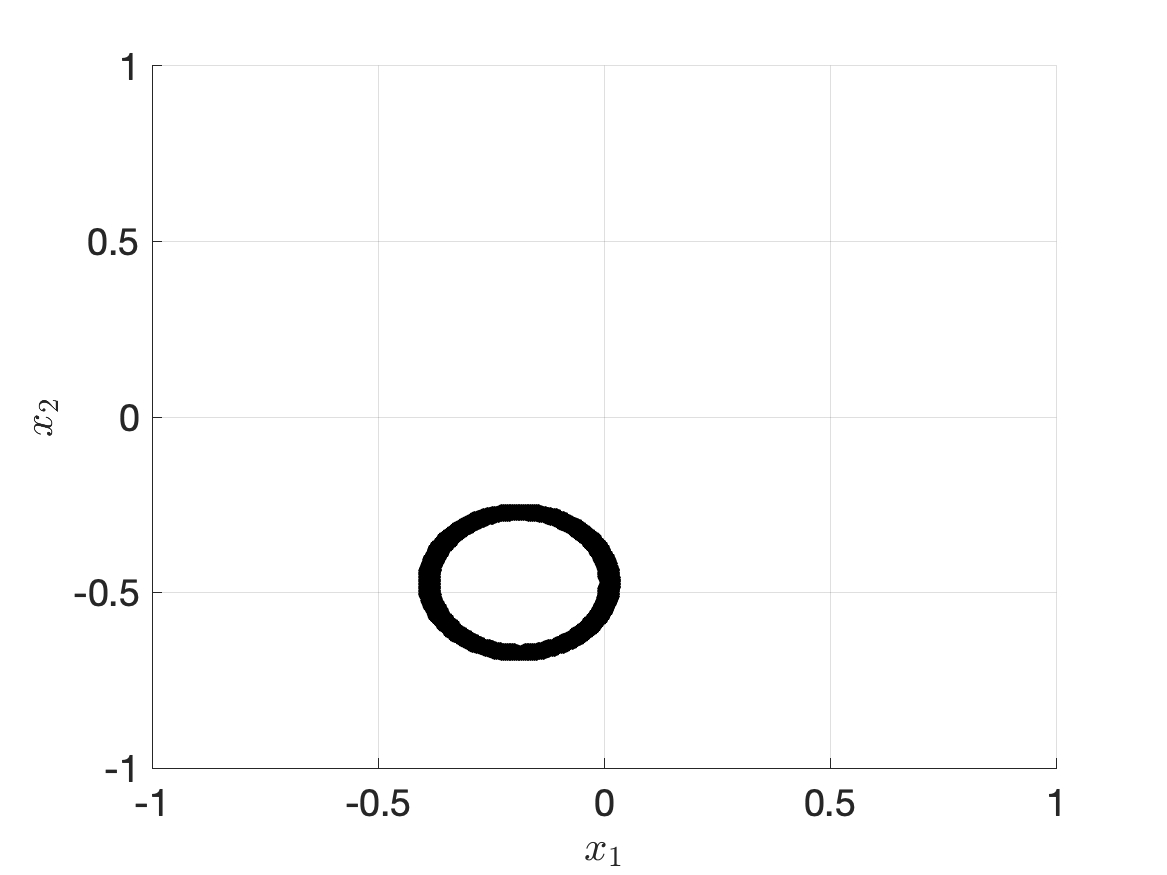}}
\hfill
\subfloat [Residual for $z = (0.5,4.35)$.]{
\includegraphics[width=2.3in]{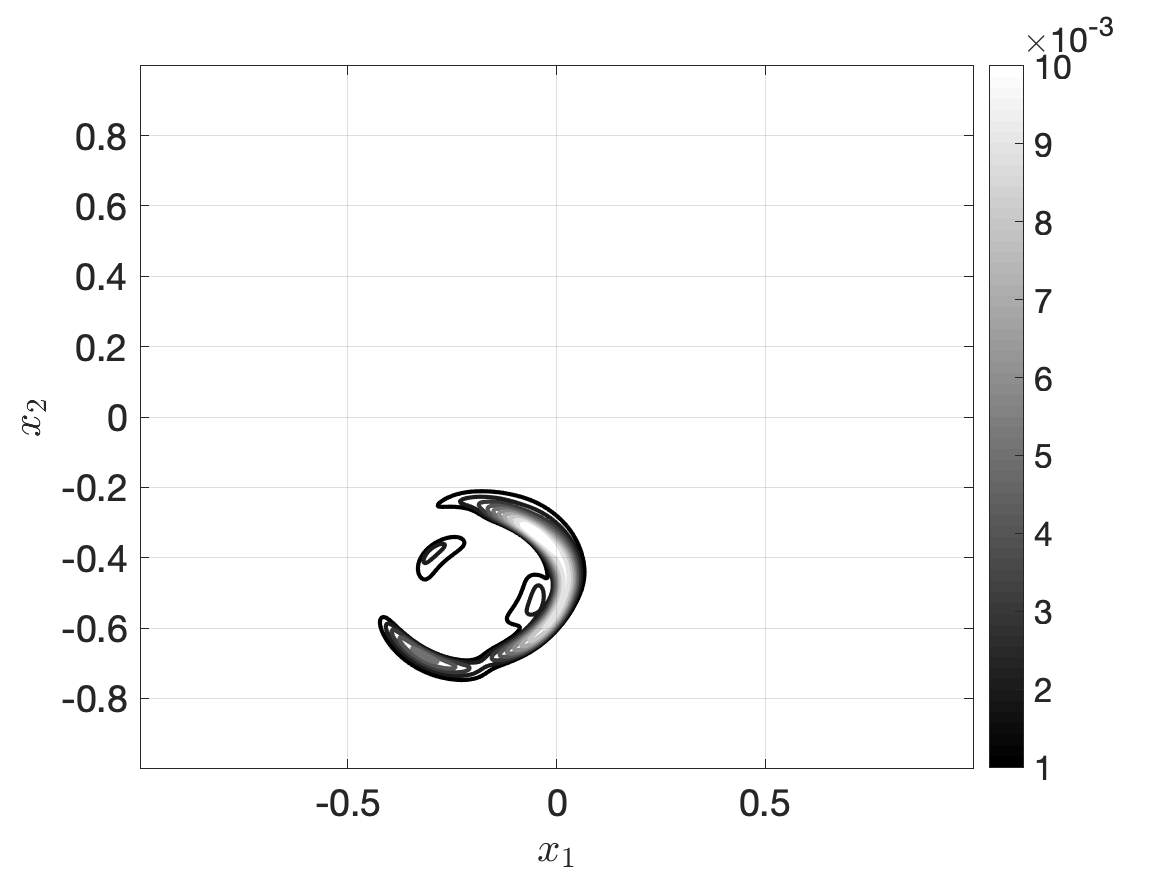}}
\caption{\textit{Results for test case-3. Results computed with $N_\mu = N_t = 6$ and $N_x = 500$. }}\label{test-3: reduced mesh}
\end{figure} 

\section{Conclusions}
We propose an online-adaptive hyper-reduction technique for the nonlinear reduced order modelling of transport dominated problems. Our nonlinear approximation space is a span of shifted snapshots and we seek a solution using residual minimization. Through a cost analysis, we conclude that residual minimization is (at least) as expensive as the full-order model. To reduce the cost of residual minimization, we perform residual minimization over a reduced mesh i.e., over a subset of the full mesh. Using numerical and analytical examples we show that, similar to the solution, the residual exhibits a transport-type behaviour. This makes the use of a fixed parameter-independent reduced mesh both inaccurate and inefficient. To account for the transport-type behaviour of the residual, we introduce online-adaptivity in the reduced mesh.

Empirically, we establish that for the same size of the reduced mesh, the online-adaptive technique greatly outperforms a non-adaptive technique. For a sufficiently small reduced mesh---almost $1\%$ to $2\%$ the  size of the full mesh---the adaptive technique provides reasonable accuracy. In contrast, for such small sizes of the reduced mesh, the non-adaptive technique lead to an unstable reduced-order model, resulting in oscillations and extremely large error values. Nonetheless, at the expense of a high computational cost, increasing the size of the reduced mesh to about $20\%$-$50\%$ of the total mesh size improved the accuracy of the non-adaptive technique. 

\appendix
\section{Example demonstrating the role of $\varphi$}\label{app: example}
\begin{example}[Introducing regularity via $\varphi$]\label{example}
Consider the set $\mcal G:=\{g(\cdot,z)\hsp :\hsp \mu\in Z\}\subset L^2(\mbb R)$, where $g(\cdot,z)$ is a step function that scales and shifts to the right, and is given as
\begin{gather*}
g(x,z):=\begin{cases}
1+z,&\hsp x\leq z \\
0,&\hsp x > z
\end{cases},\hspB z\in\mcal Z := [0,1].
\end{gather*}
For every $x\in \Omega$, $g(x,\cdot)$ belongs to $H^{1/2}(\mcal Z)$, which prohibits its approximability in a linear space. Indeed, one can prove that the Kolmogorov m-width of $\mcal G$ scales as $\mcal O(1/\sqrt{m})$---see \cite{PeterBook} for a proof.

For some $\hat z\in\mcal Z$, consider the set $\mcal G_\varphi$ that consists of all the step functions shifted such that their discontinuities are aligned with the discontinuity in $g(\cdot,\hat z)$
\begin{align*}
\mcal G_\varphi : = &\{g(\varphi(\cdot,z,\hat z),\hat z)\hsp :\hsp \varphi(x,z,\hat z) = x-(z-\hat z),\hsp\hat z\in\mcal Z\},\\= &\{(1+\hat z) g(\cdot,z)\hsp :\hsp \hat z \in \mcal Z\}.
\end{align*}
We conclude that, for all $x\in\Omega$, the function $g(\varphi(x,\cdot,\hat z),\hat z)$ is smooth. Furthermore, one can conclude that $\mcal G_{\varphi}$ is contained in the span of any single function taken from $\mcal G_{\varphi}$. Thus, for $m\geq 1$, its Kolmogorov m-width is zero. 
\end{example}

\subsection{Snapshots of the shift}\label{sec: snap sptransf}
As mentioned in the introduction, we want $u(\varphi(x,z,\cdot),\cdot)$ to be sufficiently regular. To this end, for all $z^{(i)},z^{(j)}\in \mcal Z$, one seeks to (at least approximately) satisfy the matching condition 
\begin{gather}
\varphi(\mcal D(z^{(i)}),z^{(j)},z^{(i)}) = \mcal D(z^{(i)}), \label{match feature}
\end{gather}
where $\mcal D(z^{(i)})\subset\Omega$ and $\mcal D(z^{(j)})\subset\Omega$ represent the point/curve/surface of discontinuity in $u(\cdot,z^{(i)})$ and $u(\cdot,z^{(j)})$, respectively. With our spatial shift ansatz for $\varphi$ given in \eqref{ansatz varphi}, the matching condition transforms to
\begin{gather}
\mcal D(z^{(i)}) = \mcal D(z^{(j)}) + c(z^{(i)},z^{(j)}).\label{shift via match}
\end{gather}
To approximate $\mcal D(z^{(j)})$ (and $\mcal D(z^{(i)})$), we apply to $u_N(\cdot,z^{(i)})$ the multi-resolution-analysis (MRA) based troubled cell indicator proposed in \cite{MRAdetect2014}---any other shock-detection technique (for instance, from \cite{ShockSupConvg,ShockNN}) also suffices.

In general, the above condition is (very) restrictive. For instance, in a 1D spatial domain, the condition holds only if the solution has a single shock or multiple shocks that move with the same velocity. However, the condition is violated for two shocks moving with different velocities. Furthermore, in a multi-dimensional setting, even a single shock that changes in length violates the above condition. 

Despite the restrictions of the above condition, empirically, we observe that a spatial shift provides accurate results for problems that approximately satisfy the above relation. For instance, in a multi-dimensional setting, $\mcal D(z^{(i)})$ might be a translation of $\mcal D(z^{(j)})$ but with an elongation. If the elongation is not significant then, we recover a reasonable snapshot transformation via shifting. Similarly, when $\mcal D(z^{(i)})$ has the same length as $\mcal D(z^{(j)})$ but is both a translation and a rotation of $\mcal D(z^{(j)})$, we expect a spatial shift to provide reasonable results  if the rotation is not significant.

For the above reasons, we do not strictly impose the discontinuity matching condition given in \eqref{shift via match}. Rather, we develop upon the L2-minimization technique proposed in \cite{Welper2017,RimTR,Rowley2001,Nair,RegisterMOR}. We define the set $\mcal B(z^{(j)},z^{(i)})\subset\mbb R^d$ that contains all possible shifts that match the respective points in $\mcal D(z^{(j)})$ and $\mcal D(z^{(i)})$. Equivalently,
\begin{gather}
\mcal B(z^{(j)},z^{(i)}):=\{c^*\hsp :\hsp c^* = x^*_i-x^*_j,\hsp x^*_i\in \mcal D(z^{(i)}),\hsp x^*_j\in \mcal D(z^{(j)})\}. 
\end{gather}
Out of all the possible shifts in $\mcal B(z^{(j)},z^{(i)})$, we select the one that solves the minimization problem
\begin{equation}
\begin{gathered}
c(z^{(j)},z^{(i)}) =\argmin_{c^*\in \mcal B(z^{(j)},z^{(i)})} \|u_N(\Theta[c^*],z^{(i)})-u_N(\cdot,z^{(j)})\|_{L^2(\Omega)}. \label{def c}
\end{gathered}
\end{equation}
Above, $\Theta[c^*](x) = x-c^*$ and is as given in \eqref{ansatz varphi}. We solve the above problem via enumeration. In our numerical experiments, the set $\mcal B(z^{(j)},z^{(i)})$ is not too large and a solution via enumeration is affordable.

\section{Shifting residuals}\label{app: shift snap}
The definition of the residual given in \eqref{def residual} provides
\begin{equation}
\begin{aligned}
\mcal T[-c(t_{k+1},\mu,z_{ref})]&\resFV(U_m(t_{k+1},\mu),U_m(t_{k},\mu))\\
& = \mcal T[-c(t_{k+1},\mu,z_{ref})]\left(U_m(t_{k+1},\mu)-(U_m(t_{k},\mu) + \Delta t\times \mcal F(U_m(t_{k},\mu)))\right).
\end{aligned}
\end{equation}
We simplify the different terms appearing on the right. Using the matrix-product form of the reduced snapshot $U_m(t_{k+1},\mu)$, we find 
\begin{equation}
\begin{aligned}
&\mcal T[-c(t_{k+1},\mu,z_{ref})]U_m(t_{k+1},\mu)\\
 = &\sum_{i}\alpha_i(t_{k+1},\mu) \mcal T[-c(t_{k+1},\mu,z_{ref})]\mcal T[c(t_{k+1},\mu,\zRef{i})]U(\zRef{i})\\
=&\sum_{i}\alpha_i(t_{k+1},\mu) \mcal T[c(t_{k+1},\mu,\zRef{i})-c(t_{k+1},\mu,z_{ref})]U(\zRef{i})\\
=&\sum_{i}\alpha_i(t_{k+1},\mu) \mcal T[c(z_{ref},t_{k+1},\mu) + c(t_{k+1},\mu,\zRef{i})]U(\zRef{i})\\
=&\sum_{i}\alpha_i(t_{k+1},\mu) \mcal T[c(z_{ref},\zRef{i})]U(\zRef{i}).
\end{aligned}
\end{equation}
The third and the fourth equalities are a result of (C1) and (C2), respectively. Using the assumed $L^{\infty}$ stability of the ROM, we find
\begin{equation}
\begin{aligned}
\mcal T[-c(t_{k+1},\mu,z_{ref})]U_m(t_{k},\mu)=\sum_{i}\alpha_i(t_{k},\mu) \mcal T[c(z_{ref},\zRef{i})]U(\zRef{i}) + \mcal O(1).
\end{aligned}
\end{equation}
The commutation property (C3) simplifies the $\mcal F$ term and provides
\begin{equation}
\begin{aligned}
&\Delta t\times \mcal T[-c(t_{k+1},\mu,z_{ref})]\mcal F(U_m(t_{k},\mu)))\\
& = \Delta t\times \mcal F\left(\sum_{i}\alpha_i(t_{k},\mu) \mcal T[c(z_{ref},\zRef{i})]U(\zRef{i}) + \mcal O(1)\right). 
\end{aligned}
\end{equation}
Since $\mcal F\in W^{1,\infty}(\mbb R^N)$, we find the desired result. 
\newpage
\bibliographystyle{apa}
\bibliography{papers}

\end{document}